\documentclass[a4paper,12pt,leqno]{amsart}
\usepackage{latexsym}
\usepackage[all]{xy}

\usepackage{amssymb} 
\usepackage{amsmath} 
\usepackage{color}
\usepackage{comment}

\usepackage{amscd}

\def\Z{{\mathbb{Z}}}
\def\CC{{\mathbb{C}}}
\def\RR{{\mathbb{R}}}
\def\R{{\mathbb{R}}}
\def\A{{\mathcal{A}}}

\def\CR{{\mathcal R}} 
\def\MCD{{\mathcal D}} 
\def\Y{Y}
\def\s{q}

\DeclareMathOperator{\Der}{Der}

\DeclareMathOperator{\Poin}{Poin}

\DeclareMathOperator{\Hess}{Hess}

\numberwithin{equation}{section}

\theoremstyle{break}
 \newtheorem{theorem}{Theorem}[section]
 \newtheorem{prop}[theorem]{Proposition}
 \newtheorem{cor}[theorem]{Corollary}
 \newtheorem{lemma}[theorem]{Lemma}
 
 \newtheorem{conj}[theorem]{Conjecture}

 \theoremstyle{definition}
 \newtheorem{define}[theorem]{Definition}
 \newtheorem{rem}[theorem]{Remark}
 \newtheorem{example}[theorem]{Example}

\title{Hessenberg varieties and hyperplane arrangements}

\author{
Takuro Abe,
\address{
Institute of Mathematics for Industry, 
Kyushu University,
Fukuoka 819-0395, Japan.}
\email{abe@imi.kyushu-u.ac.jp}
Tatsuya Horiguchi,\address{Osaka City University Advanced Mathematical Institute, 3-3-138 Sugimoto, Sumiyoshi-ku, Osaka 558-8585, Japan.}
\email{tatsuya.horiguchi0103@gmail.com}
Mikiya Masuda,\address{Department of Mathematics, Osaka City University, Sumiyoshi-ku, Osaka 558-8585, Japan.}
\email{masuda@sci.osaka-cu.ac.jp}\\
Satoshi Murai,\address{
Department of Pure and Applied Mathematics,
Graduate School of Information Science and Technology,
Osaka University,
Suita, Osaka, 565-0871, Japan}
\email{s-murai@ist.osaka-u.ac.jp}
and Takashi Sato\address{Department of Mathematics, Kyoto University, Kyoto, 606-8502, Japan}
\email{t-sato@math.kyoto-u.ac.jp}
}
\date{\today} 

\begin{document}

\maketitle

\begin{abstract}

Given a semisimple complex linear algebraic group $G$ and a lower ideal $I$ in positive roots of $G$, three objects arise:
the ideal arrangement $\mathcal{A}_I$, the regular nilpotent Hessenberg variety $\mbox{Hess}(N,I)$, and the regular semisimple Hessenberg variety $\mbox{Hess}(S,I)$.
We show that 
a certain graded ring derived from the logarithmic derivation module of $\mathcal{A}_I$ is isomorphic to
$H^*(\mbox{Hess}(N,I))$ and $H^*(\mbox{Hess}(S,I))^W$, 
the invariants in $H^*(\Hess(S,I))$ under an action of the Weyl group $W$ of $G$. 
This isomorphism is shown 
for general Lie type, 
and generalizes Borel's celebrated theorem showing that the coinvariant algebra of $W$ is isomorphic to the cohomology ring of the flag variety $G/B$. 

This surprising connection between Hessenberg varieties and hyperplane
arrangements enables us to produce a number of interesting
consequences.  For instance, the surjectivity of the restriction map
$H^*(G/B)\to H^*(\mbox{Hess}(N,I))$ announced by Dale
Peterson 
and an affirmative answer to 
a conjecture of Sommers-Tymoczko are immediate consequences.  We also
give an explicit ring presentation of $H^*(\mbox{Hess}(N,I))$ in
types $B$, $C$, and $G$.  Such a presentation was already known in type
$A$ or when $\Hess(N,I)$ is the Peterson variety.  Moreover, we find
the volume polynomial of $\mbox{Hess}(N,I)$ and see that the hard
Lefschetz property and the Hodge-Riemann relations hold for
$\mbox{Hess}(N,I)$, despite the fact that it is a singular variety in general.
\end{abstract}

\section{Introduction}

In this paper we study Hessenberg varieties and hyperplane
arrangements.  Hessenberg varieties are subvarieties of a flag
variety. Their geometry and (equivariant) topology have been studied
extensively since the late 1980s (\cite{dMPS,dMS}).  This subject lies
at the intersection of, and makes connections between, many research
areas such as geometric representation theory, combinatorics,
algebraic geometry, and topology (see the references in \cite{AHHM}). 
More recently, a remarkable connection to graph theory has been found
(\cite{SW}).  Hyperplane arrangements are collections of finitely many
hyperplanes.  Although these are conceptually simple objects, it is a
subject actively studied from various 
viewpoints such as algebraic geometry,
topology, representation theory, combinatorics, statistics, and so on
(cf. for example
\cite{ABCHT,OT,T,T2,Z}).
These two objects -- Hessenberg varieties and hyperplane arrangements
-- may seem unrelated at first glance,
but our results connect their topology and algebra.  

In the
following, all cohomology groups will be taken with real coefficients
unless otherwise specified but the results actually hold with rational
coefficients under suitable modification.

We begin with a prototype of our results.
Let $G$ be a semisimple complex linear algebraic group of rank $n$, $B$ a Borel subgroup, and $T$ the maximal torus in $B$.
Let $\hat T$ be the group of characters of $T$ and $\CR$ the symmetric algebra $\mbox{Sym}(\hat T\otimes\R)$ of $\hat T\otimes \R$, where $\hat T$ is regarded as an additive group.
The Weyl group $W$ of $G$ acts on $\hat T$ and $\CR$. 
To each $\alpha\in \hat T$, one can associate a complex line bundle
$L_\alpha$ over the flag variety $G/B$; assigning its Euler class $e(L_\alpha)$ to $\alpha$ induces a ring homomorphism 
\begin{equation} \label{eq:1}
\varphi\colon \CR=\mbox{Sym}(\hat T\otimes\R) \to H^*(G/B)
\end{equation}
which doubles the grading on $\CR$.
Borel's celebrated theorem \cite{B} states that $\varphi$ is surjective and its kernel is the ideal $(\CR^W_+)$ generated by the $W$-invariants in $\CR$ with zero constant term, so that $\varphi$ induces an isomorphism 
\[ \CR/(\CR^W_+)\cong H^*(G/B). \]

The ideal $(\CR^W_+)$ also appears in the study of hyperplane arrangements.
Let $\mathfrak{t}$ denote the vector space dual to $\hat T\otimes\R$, so that $\hat T\otimes \R=\mathfrak{t}^*$.
We may think of $\mathfrak{t}$ as the Lie algebra of the maximal compact torus in $T$.
A root of $G$ is a linear function on $\mathfrak{t}$ and the hyperplane arrangement 
\[
\A_{\Phi^+}:=\{\ker\alpha\mid \alpha\in \Phi^+\} \qquad \text{($\Phi^+$: the set of positive roots of $G$)}
\]
is called the Weyl arrangement.
The logarithmic derivation $\CR$-module $D(\A_{\Phi^+})$ of
$\A_{\Phi^+}$ (geometrically speaking, this consists of polynomial vector fields on $\mathfrak{t}$ tangent to $\A_{\Phi^+}$) is defined in $\CR\otimes \mathfrak{t}$. Then 
for a $W$-invariant non-degenerate quadratic form $Q \in \CR^W$, the ideal $\{\theta(Q) \mid \theta \in D(\A_{\Phi^+})\}$ turns out to be 
the ideal $(\CR^W_+)$.
See Theorem \ref{coinv} for details.     

Our aim in this paper is to generalize the phenomenon described
above. 
For that we need one more piece of data:
a lower ideal in $\Phi^+$, that is, a lower closed subset of $\Phi^+$ with respect to the usual partial order on $\Phi^+$.
To such a lower ideal $I$ of $\Phi^+$, we associate the \textbf{ideal
  arrangement} $\A_{I}$ \textbf{of $I$}, defined to be the subarrangement  
\[
\A_{I}:=\{\ker\alpha\mid \alpha\in I\}
\] 
of $\A_{\Phi^+}$. 
Then a graded ideal of $\CR$, denoted by $\mathfrak{a}(I)$, can be defined for $\A_I$ similarly for $\A_{\Phi^+}$, 
i.e., $$\mathfrak{a}(I):=\{\theta(Q) \mid \theta \in D(\A_I)\}.$$ 
%
In particular, $\mathfrak{a}(\Phi^+)=(\CR^W_+)$.  
The ideal $\mathfrak{a}(I)$ plays a key role in our argument and will be discussed in a more general setup in \cite{AMN}. 

On the other hand, Hessenberg varieties are defined as follows.
Let $\mathfrak{g}$ (resp.\ $\mathfrak{b}$) be the Lie algebra of $G$ (resp.\ $B$) and $\mathfrak{g}_\alpha$ the root space of a root $\alpha$.
The subspace $H(I)=\mathfrak{b}\oplus (\bigoplus_{\alpha \in I} \mathfrak{g}_{-\alpha})$ of $\mathfrak{g}$ is  $\mathfrak{b}$-stable and, for $X \in \mathfrak{g}$, the \textbf{Hessenberg variety} $\Hess (X,I)$ is defined by 
\[
\Hess (X,I):=\{g B \in G/B\mid \mbox{Ad}(g^{-1})(X) \in H(I)\}.
\]
When $I$ is empty (so $H(I)=\mathfrak{b}$) and $X$ is nilpotent,
$\Hess(X,I)$ is the famous Springer variety (or Springer fiber) and
has been
studied by many people in connection with geometric representation
theory (see e.g.\ the survey article \cite{Ty4}).

The Hessenberg variety $\Hess(X,I)$ is called \textbf{regular
  nilpotent} (resp.\ \textbf{regular semisimple}) if $X$ is regular
nilpotent (resp.\ regular semisimple).  These two cases have been much
studied 
in recent research on Hessenberg varieties. 
In particular, affine pavings have been constructed in
these cases, from which it follows that their odd degree cohomology groups vanish, and
their even degree Betti numbers are well understood (\cite{dMPS,Pr1,Ty}).  However, their cohomology ring structures are
not well understood in general.
In this paper we show that the study of these cohomology rings is
closely related to the logarithmic derivation modules of hyperplane
arrangements.  This surprising connection enables us to produce a
number of interesting consequences and provides a systematic method to
give an explicit presentation of the cohomology ring of a regular
nilpotent Hessenberg variety.  Moreover, our argument is
independent of Lie type, i.e., we do not use the classification of root
systems except in the explicit computation of the ring structure for
specific Lie types.


First we treat the nilpotent case.
We denote by $N$ a regular nilpotent element of $\mathfrak{g}$.
Since the regular nilpotent Hessenberg variety $\Hess(N,I)$ is a subvariety of $G/B$, $\varphi$ in \eqref{eq:1} followed by the restriction map yields a homomorphism
\[
\varphi_I\colon \CR\to H^*(\Hess(N,I))
\] 
which doubles the grading on $\CR$.
Our first main theorem is the following.

\begin{theorem} \label{nilpotentmain}
The map $\varphi_I$ is surjective and its kernel is $\mathfrak{a}(I)$. Hence $\varphi_I$ induces an isomorphism 
\begin{equation*}\label{Borel}
\CR/\mathfrak{a}(I)\cong H^*(\Hess(N,I)).
\end{equation*}
\end{theorem}

When $I=\Phi^+$, 
Theorem \ref{nilpotentmain} is nothing but Borel's isomorphism between the coinvariant ring and $H^*(G/B)$ mentioned above. 
Theorem \ref{nilpotentmain} has two important corollaries.
One is Corollary~\ref{surjection} below, which was announced by Dale Peterson 
(see \cite[Theorem 3]{BC}).  

\begin{cor} \label{surjection} 
The restriction map
  $H^*(G/B)\to H^*(\Hess (N,I))$ is surjective and $H^*(\Hess(N,I))$
  is a complete intersection, and in particular, is a Poincar\'e duality
  algebra.  Moreover, the Poincar\'e polynomial of $\Hess(N,I)$ is
  given by the product formula
\begin{equation} \label{eq:PoinHessNI}
\mbox{Poin}(\Hess(N,I),\sqrt{q})=\prod_{\alpha\in I}\frac{1-q^{ht(\alpha)+1}}{1-q^{ht(\alpha)}},
\end{equation}
where $ht(\alpha)$ denotes the sum of the coefficients of $\alpha$ over the simple roots. 
\end{cor}

The other corollary is an affirmative answer to the first conjecture
in \cite{STy} by Sommers-Tymoczko. The ideal arrangement $\A_I$ is
known to be free, that is, the logarithmic derivation $\CR$-module
$D(\A_I)$ is free as a graded $\CR$-module. The free $\CR$-module
$D(\A_I)$ has $n$ homogeneous generators, where $n$ is the rank of
$G$, and their (polynomial) degrees are called the exponents of
$\A_I$. See \S2.1 for details.

\begin{cor} \label{Weyltype} 
Let $\mathcal{W}^I$ be the subsets of $I$
  of Weyl type (see \S 6) and let $d_1^I,\dots,d_n^I$ be the exponents
  of $\A_I$.  Then
\begin{equation}
\sum_{\Y \in \mathcal{W}^I} \s^{|\Y|}=\prod_{i=1}^n (1+\s+\cdots+\s^{d_i^I}).
\label{eq00}
\end{equation}
\end{cor}

\noindent
Corollary~\ref{Weyltype} was posed as a conjecture, and verified by Sommers and Tymoczko in \cite{STy} for types 
$A,\ B,\ C$ and $G_2$ by inductive arguments, and $F_4$ and $E_6$ by 
using computer calculation (see \cite[Theorem 4.1]{STy}). Also,
Schauenburg confirmed Corollary~\ref{Weyltype} for types $D_5,\ D_6,\ D_7$ and $E_7$ by 
direct computation, and 
R\"ohrle confirmed Corollary~\ref{Weyltype} for type $D_4$, and type $E_8$ with 
some possible exceptions by combining inductive argument 
with direct computation (see \cite[Theorem 1.28]{Roe}).

We prove Corollary~\ref{Weyltype} by connecting 
Weyl type subsets to the geometry of Hessenberg varieties, which gives us a 
classification-free proof of Corollary~\ref{Weyltype}.
Explicitly, we show Corollary \ref{eq00} by computing the Poincar\'e
polynomial of $\Hess(N,I)$ in two ways.  Indeed, the left hand side of
\eqref{eq00} is obtained from the affine paving of $\Hess(N,I)$
(\cite{Pr1,Ty}), while the right hand side of \eqref{eq00}
agrees with \eqref{eq:PoinHessNI}.

Next we treat the semisimple case.  We denote by $S$ a regular
semisimple element of $\mathfrak{g}$.  The regular semisimple
Hessenberg variety $\Hess(S,I)$ has different features from
$\Hess(N,I)$.  For instance, $\Hess(S,I)$ is smooth and invariant
under the $T$-action on $G/B$, while $\Hess(N,I)$ is in general
neither. 
Moreover, the composition of the homomorphisms
\[ \psi_I\colon \CR\to H^*(\Hess(S,I)) \]
obtained from $\varphi$ in \eqref{eq:1} with the restriction
map is not surjective.  In general $\Hess(S,I)$ does not have a
natural action of the Weyl group $W$, but the cohomology
$H^*(\Hess(S,I))$ does admit a natural $W$-action, as noticed by Tymoczko \cite{Ty2} (in type
$A$).  Indeed, $\Hess(S,I)$ is a GKM manifold and its associated GKM
graph has an action of $W$, and it induces a linear action of $W$ on
$H^*(\Hess(S,I))$.  Our second main theorem is the following.

\begin{theorem} \label{ssmain}
The image of $\psi_I$ is 
the ring of $W$-invariants $H^*(\Hess(S,I))^W$ 
and the kernel of $\psi_I$ is $\mathfrak{a}(I)$. Hence $\psi_I$ induces an isomorphism 
\[ \CR/\mathfrak{a}(I)\cong H^*(\Hess (S,I))^W.\]
\end{theorem}

Theorems~\ref{nilpotentmain} and~\ref{ssmain} (and
Corollary~\ref{surjection}) imply that there exists an isomorphism
between $H^*(\Hess(N,I))$ and $H^*(\Hess(S,I))^W$ which makes the
following diagram commute:
\[
  \xymatrix{
  & H^*(G/B) \ar[dl] \ar[dr]& \\
  H^*(\Hess(N,I)) \ar[rr]^{\cong}& & H^*(\Hess(S,I))^W
  }
\]
where the slanting arrows are restriction maps which are surjective.
We note that when $I$ consists of all simple roots, $\Hess(N,I)$ is
called the Peterson variety and $\Hess(S,I)$ is a toric variety, and
in this case the isomorphism between $H^*(\Hess(N,I))$ and
$H^*(\Hess(S,I))^W$ can be observed from their explicit ring
presentations (\cite{HHM,K}).  The isomorphism is also known in
type $A$ for any $I$, as was shown in \cite{AHHM}.  In fact, in
\cite{AHHM} the authors find an explicit
ring presentation of $H^*(\Hess(N,I))$ in type $A$ and construct
the isomorphism explicitly using it.
In constrast, in this manuscript, our result holds for any type and any $I$, our
proofs do not depend on the classification of Lie types, 
and we do not need an explicit ring presentation of $H^*(\Hess(N,I))$
in order to show the above isomorphism. 

The graded ring $\CR/\mathfrak{a}(I)$ is a complete intersection, so
that $\mathfrak{a}(I)$ is generated by $n$ homogeneous elements that
come from an $\CR$-basis of $D(\A_I)$.  Another advantage of our
approach is
that 
one can check whether $n$ elements in $D(\A_I)$ form generators using
Saito's criterion for the freeness of logarithmic derivation modules.
In fact, we carry out this idea and find explicit generators of
$\mathfrak{a}(I)$ in types $B$, $C$, and $G$, so that we obtain an
explicit ring presentation of $H^*(\Hess(N,I))$ in these Lie types.
Such a presentation was known in type $A$ (\cite{AHHM}) or when
$\Hess(N,I)$ is the Peterson variety (\cite{fu-ha-ma,HHM}).
Our method is also applicable to these cases and reproves their
results.

The organization of this paper is as follows. In \S 2 we collect some
general results and notions on hyperplane arrangements, commutative
algebra, and equivariant cohomology.  In \S 3 we mention our setting,
give the precise definition of the ideal $\mathfrak{a}(I)$, and recall
some results on hyperplane arrangements and Hessenberg varieties
necessary for our purposes.  In \S 4 we describe the Hilbert series of
$\CR/\mathfrak{a}(I)$ in terms of the ideal exponents of $\A_I$ and
prove a key property which the ideals $\mathfrak{a}(I)$ satisfy.  In
\S 5 we study regular nilpotent Hessenberg varieties using equivariant
cohomology.  We introduce the ideal $\mathfrak{n}(I)$ for each $I$ and
show that those ideals satisfy similar properties to
$\mathfrak{a}(I)$'s.  In \S 6 we study the relation between the set
$\mathcal{W}^I$ of the Weyl type subsets of $I$ and chambers of the
ideal arrangement $\A_I$.  In \S 7 we prove Theorem
\ref{nilpotentmain} by showing that $\mathfrak{n}(I)=\mathfrak{a}(I)$
and deduce a few corollaries, especially the affirmative answer to the
conjecture by Sommers-Tymoczko. We study regular semisimple Hessenberg
varieties using GKM theory in \S 8, and prove Theorem \ref{ssmain} in
\S 9.  In \S 10 we give explicit generators of $\mathfrak{a}(I)$ in
types $B$, $C$, and $G$. Those explicit generators are analogs of
the generators in type $A$ found by \cite{AHHM}.  In \S 11 we find
the volume polynomial of $\Hess(N,I)$, i.e., the polynomial on
$\mathfrak{t}$ whose annihilator in the ring of differential operators
can be identified with $\mathfrak{a}(I)$.  In \S 12 we will see that
the hard Lefschetz property and the Hodge-Riemann relations hold for
$\Hess(N,I)$ although it is a singular variety in general.  \medskip

\noindent
\textbf{Acknowledgements}. 
We are grateful to Hiraku Abe for fruitful discussions and comments on this paper and Naoki Fujita for his stimulating question. We are grateful to Hiroaki Terao for the discussion on the free basis for 
height subarrangements of type $A$. 
We are also grateful to Megumi Harada for her comments on the paper.
The first author is partially supported by JSPS Grant-in-Aid for Scientific Research (B) 16H03924. 
The second author is partially supported by JSPS Grant-in-Aid for JSPS Fellows 15J09343. 
The third author is partially supported by JSPS Grant-in-Aid for Scientific Research (C) 16K05152.

\section{Preliminaries}

In this section we collect some results and review some notions on hyperplane arrangements, commutative algebra, and equivariant cohomology, which will be used throughout this paper.  

\subsection{Hyperplane arrangements} \label{subsect:HA}

In this subsection we  review definitions and results on hyperplane arrangements. For 
general reference, see \cite{OT}.

Let $V$ be a finite dimensional real vector space and let $\A$ be a \textbf{hyperplane arrangement} in $V$, i.e., a 
finite set of linear hyperplanes in $V$. Let $M(\A):=V \setminus \bigcup_{H \in \A} H$ and define the set of \textbf{chambers} $C(\A)$ of $\A$ by 
\[
C(\A):=\{\mbox{connected components of }M(\A)\}.
\]
The polynomial $\pi(\A,t):=\Poin(M(\A) \otimes_{\RR} \CC,t)$ is called 
the \textbf{Poincar\'e polynomial} of $\A$. It is well-known that 
$|\A|$ coincides with the coefficient of $t$ in $\pi(\A,t)$, and $\pi(\A,t)$ depends only on the combinatorial 
structure of $\A$, see \cite{OS} for details. 
Now let us introduce several results used for the proofs of main results in this paper. The following is a 
well-known counting result of chambers by Zaslavsky.


\begin{theorem}[\cite{Z}]
$|C(\A)|=\pi(\A,1)$.
\label{Zaslavsky}
\end{theorem}

Let $R$ be the symmetric algebra $\mbox{Sym}(V^*)$ of $V^*$ the dual space to $V$. 
An element of $V$ is a linear function on $V^*$ and extends to a derivation on $R$:
\[
v(fg)=v(f)g+fv(g)\quad \text{for $v\in V,\ f,g\in R$}.
\]
We then define the $R$-module of derivations on $R$ by 
\[
\Der R:=R\otimes V. 
\]
Choosing a linear coordinate system $x_1,\dots, x_n$ on $V$, i.e., $x_1,\ldots,x_n$ is a basis for $V^*$, $\Der R$ can be expressed as $\bigoplus_{i=1}^n R (\partial/\partial{x_i})$. 
A nonzero element $\theta\in \Der R$ is homogeneous of (polynomial) degree $d$ if $\theta=\sum_{k=1}^{\ell}f_k\otimes v_k$ $(f_k\in R,\ v_k\in V)$ and all non-zero 
$f_k$'s are of degree $d$.   

For each $H \in \A$, let 
$\alpha_H \in V^*$ be the defining linear form of $H$. The logarithmic derivation module 
$D(\A)$ of $\A$ is a graded $R$-module defined by 
\[
D(\A):=\{\theta \in \Der R \mid 
\theta(\alpha_H) \in R \alpha_H \ (^\forall H \in \A)\}.
\]
In general $D(\A)$ is reflexive but not necessarily free.
We say that $\A$ is \textbf{free} with $\exp(\A)=
(d_1,\ldots,d_n)$ if $D(\A)$ is a free $R$-module with homogeneous 
basis $\theta_1,\ldots,\theta_n$ of degree $d_1,\ldots,d_n$.

\begin{theorem}[Terao's factorization, \cite{T2}]
\label{factorization}
Let $\A$ be free with $\exp(\A)=(d_1,\ldots,d_n)$. Then 
$\pi(\A,t)=\prod_{i=1}^n (1+d_i t)$. In particular, $|C(\A)|=\prod_{i=1}^n (1+d_i)$ by Theorem~\ref{Zaslavsky}. 
\end{theorem}

We finally recall the following well-known criterion for bases of logarithmic derivation modules of arbitrary hyperplane arrangement.

\begin{theorem}[Saito's criterion, \cite{S2}, see also \cite{OT}]
Let $\A$ be a hyperplane arrangement in $V$ and let $\theta_1,\ldots,\theta_n \in D(\A)$ be homogeneous derivations. Then $\theta_1,\ldots,
\theta_n$ form an $R$-basis for $D(\A)$ if and only if 
$\theta_1,\ldots,\theta_n$ are $R$-independent and $\sum_{i=1}^n \deg \theta_i=|\A|$.
\label{Saito}
\end{theorem}


\subsection{Poincar\'e duality algebras and complete intersections} \label{subsec:CP}

Here we introduce some basic algebraic properties of Poincar\'e duality algebras and complete intersections.

A graded $\R$-algebra $A=R/\mathfrak{a}$ is {\bf Artinian} if $\dim_\R A < \infty$.
Let $A=A_0 \oplus A_1 \oplus \cdots \oplus A_r$ be an Artinian graded $\R$-algebra, where $A_i$ is the homogeneous component of $A$ of degree $i$ and $A_r$ is non-zero.
The algebra $A$ is said to be a {\bf Poincar\' e duality algebra} of socle degree $r$ if $A_r \cong \R$ and the map
\[
A_i \times A_{r-i} \to A_r,
\qquad (a,b) \mapsto ab
\]
is non-degenerate.
For an ideal $\mathfrak{a} \subset R$ and $ f \in R$,
let $\mathfrak{a}:f$ be the ideal of $R$ defined by
\[
\mathfrak{a}:f=\{ g \in R \mid fg \in \mathfrak{a}\}.
\]
We notice that, for ideals $\mathfrak{a},\mathfrak{a}' \subset R$,
the multiplication map $R/ \mathfrak{a}' \stackrel{\times f} \to R/\mathfrak{a}$ is well-defined if and only if $\mathfrak{a}' \subset \mathfrak{a}:f$.

We need the following simple algebraic property of Poincar\'e duality algebras.

\begin{lemma}
\label{colonideal}
Let $\mathfrak{a},\mathfrak{a}'$ be homogeneous ideals of $R$ and $f \in R$ a homogeneous polynomial of degree $k$ with $f \not \in \mathfrak{a}$.
Suppose $\mathfrak{a}' \subset \mathfrak{a}:f$.
If $R/\mathfrak{a}'$ is a Poincar\'e duality algebra of socle degree $r$
and $R/\mathfrak{a}$ is a  Poincar\'e duality algebra of socle degree $r+k$, then $\mathfrak{a}'=\mathfrak{a}:f$.
\end{lemma}

\begin{proof}
Since $\mathfrak{a}' \subset \mathfrak{a}:f$, we have a natural $R$-homomorphism
\[
R/\mathfrak{a}' \stackrel{ \times f} \longrightarrow R/\mathfrak{a}.
\]
Observe that the above map sends $(R/\mathfrak{a}')_i$ to $(R/\mathfrak{a})_{i+k}$.
We first show $\times f: (R/\mathfrak{a}')_r \to (R/\mathfrak{a})_{r+k}$ is an isomorphism.
Since $(R/\mathfrak{a}')_r \cong (R/\mathfrak{a})_{r+k} \cong \R$, we only need to prove that the map is not the zero map.
Since $f \not \in \mathfrak{a}$, $f+\mathfrak{a} \in (R/\mathfrak{a})_k$ is non-zero.
By the Poincar\'e duality of $R/\mathfrak{a}$, there is $g \in R_r$ such that $f g +\mathfrak{a} \in (R/\mathfrak{a})_{r+k}$ is non-zero in $R/\mathfrak{a}$.
This implies that the map $\times f : (R/\mathfrak{a}')_r \to (R/\mathfrak{a})_{r+k}$ is not zero since it sends $g+\mathfrak{a}' \in (R/\mathfrak{a}')_r$ to a non-zero element $fg+\mathfrak{a} \in (R/\mathfrak{a})_{r+k}$.

We now prove that $\mathfrak{a}'=\mathfrak{a}:f$.
Let $h \not \in \mathfrak{a}'$ be a homogeneous polynomial of degree $i$. What we must prove is that $hf \not \in \mathfrak{a}$, equivalently, $hf +\mathfrak{a} \in R/\mathfrak{a}$ is non-zero.
By the Poincar\'e duality of $R/\mathfrak{a}'$, there is a polynomial $g' \in R_{r-i}$ such that $hg' + \mathfrak{a}' \in (R/\mathfrak{a}')_r$ is non-zero in $R/\mathfrak{a}'$.
Then $hg'f +\mathfrak{a} \in (R/\mathfrak{a})_{r+k}$ is non-zero since $\times f:(R/\mathfrak{a}')_r \to (R/\mathfrak{a})_{r+k}$ is an isomorphism.
Since $ hg'f+\mathfrak{a} = (hf+\mathfrak{a})g' \in R/\mathfrak{a} $, the element $hf+\mathfrak{a} \in (R/\mathfrak{a})_{i+k}$ is non-zero as desired.
\end{proof}

%

A sequence of homogeneous polynomials $f_1,\dots,f_i \in R$ of positive degrees is a {\bf regular sequence} of $R$ if $f_j$ is a non-zero divisor of $R/(f_1,\dots,f_{j-1})$ for all $j=1,2,\dots,i$.
A graded $\R$-algebra $A=R/\mathfrak{a}$ is called a {\bf complete intersection} if $\mathfrak{a}$ is generated by a regular sequence.
If $A=R/(f_1,\dots,f_i)$ is a complete intersection, then its Krull dimension is $n-i$.
Hence $A$ is Artinian if and only if $i=n$.
The following facts are well-known in commutative algebra.
See \cite[Theorem 2.3.3]{BH} and \cite[p.\ 35]{St}.

\begin{lemma}
\label{completeintersection}
If a graded $\R$-algebra $A=R/\mathfrak{a}$ is Artinian and $\mathfrak{a}$ is generated by $n$ polynomials, then $A$ is a complete intersection.
\end{lemma}

Recall that for a graded $\R$-algebra $A$,
its {\bf Hilbert series} is the formal power series
$$F(A,\s):= \sum_{i \geq 0} (\dim_\R A_i) \s^i.$$

\begin{lemma}\label{socledegree}
Let $A=R/(f_1,\dots,f_n)$ be a graded Artinian complete intersection and let $d_i = \deg f_i$ for all $i$.
Then $A$ is a Poincar\'e duality algebra of socle degree $\sum_{i=1}^n (d_i-1)$ and its Hilbert series is given by
\[
F(A,\s)= \prod_{i=1}^n (1+\s+ \cdots +\s^{d_{i}-1}).
\]
\end{lemma}

\subsection{Equivariant cohomology} \label{subsect:EC}

We shall briefly review some facts on equivariant cohomology needed later. We use \cite{Hs} as 
a reference for the results in this subsection. All cohomology groups will be taken with real coefficients. 

Let $K$ be a $\CC^*$-torus\footnote{The same argument works for a compact torus.} of rank $m$ and let $EK\to BK$ be a universal principal $K$-bundle, where $EK$ is a contractible topological space with free $K$-action and $BK=EK/K$.
In fact, we may think of $BK$ as $(\CC P^\infty)^m$.
Therefore $H^*(BK)$ is a polynomial ring in $m$ variables of degree 2.
If $K$ acts on a topological space $X$, then the equivariant cohomology of the $K$-space $X$ is defined by 
\[
H^*_K(X):=H^*(EK\times_K X), 
\]
where $EK\times_K X$ is the orbit space of $EK\times X$ by the $K$-action defined by $k(u,x)=(uk^{-1},kx)$ for $(u,x)\in EK\times X$ and $k\in K$.
If $X$ is one point $pt$, then $EK\times_K pt=BK\times pt$; so 
\[
H^*_K(pt)=H^*(BK).
\]  
More generally, if the $K$-action on $X$ is trivial, then $H^*_K(X)=H^*(BK)\otimes H^*(X)$.  

Since the $K$-action on $EK$ is free, the first projection $\pi\colon EK\times_K X\to EK/K=BK$ yields a fibration with fiber $X$.
Therefore one can regard $H^*_K(X)$ as an $H^*(BK)$-module via $\pi^*\colon H^*(BK)\to H^*_K(X)$.  We also have the restriction map $H^*_K(X)\to H^*(X)$ since $X$ is a fiber.  

Suppose that $H^{odd}(X)$ vanishes (this is satisfied in our case treated later).
Then since $H^{odd}(BK)$ also vanishes (because $BK=(\CC P^\infty)^m$), the Serre spectral sequence of the fibration $\pi\colon EK\times_K X\to BK$ collapses and hence 
\[
H^*_K(X)\cong H^*(BK)\otimes H^*(X)\quad\text{as $H^*(BK)$-modules}
\]
and the restriction map $H^*_K(X)\to H^*(X)$ is surjective.  
In addition, under some technical hypothesis on $X$ which are satisfied by the spaces considered in this paper\footnote{For instance, it would certainly suffice if $X$ is locally contractible, compact, and Hausdorff.}, 
it follows from the localization theorem (\cite[p.40]{Hs})  that the restriction map to the $K$-fixed point set $X^K$ of $X$
\[
H^*_K(X)\to H^*_K(X^K)\cong H^*(BK)\otimes H^*(X^K)
\]
is injective.
In our case, $X^K$ consists of finitely many points and $H^*_K(X^K)$ is a direct sum of copies of the polynomial ring $H^*(BK)$.  

For an oriented $K$-vector bundle $E\to X$ ($E$ is a complex vector bundle in our case treated later), one can associate a vector bundle 
\[
EK\times_K E\to EK\times_K X
\]
with the orientation induced from $E$.
The Euler class of this oriented vector bundle is called the \textbf{equivariant Euler class} of $E$ and denoted by $e^K(E)$.
Note that $e^K(E)$ lies in $H^*_K(X)$ and the restriction map $H^{*}_K(X)\to H^{*}(X)$ sends $e^K(E)$ to the (ordinary) Euler class $e(E)$ of $E$.

\section{Setting}

In this paper, we discuss three objects:
ideal arrangements, regular nilpotent Hessenberg varieties, and regular semisimple Hessenberg varieties.
They all arise from the data of a semisimple linear algebraic group (with a fixed Borel subgroup) and a lower ideal in the set of positive roots of the group.
In this section, we give the precise definition of those three objects and related notions and recall some results on them needed later.
Throughout this paper, all (co)homology groups will be taken with real coefficients unless otherwise stated. 

Let $G$ be a semisimple linear algebraic group of rank $n$.  We fix a Borel subgroup $B$ of $G$. Then the following data is uniquely determined: 

\begin{itemize} 
\item the maximal torus $T$ of $G$ in $B$
\item the Weyl group $W=N(T)/T$ where $N(T)$ is the normalizer of $T$
\item the Lie algebras $\mathfrak{b}\subset \mathfrak{g}$ of $B$ and $G$
\item $\Phi=$ \{roots of $G$\}
\item $\Phi^+=\{$positive roots in $\Phi \}$
\item the simple roots $\Delta=\{\alpha_1,\dots,\alpha_n\}$
\item the partial order $\preceq$ on $\Phi$; $\alpha \preceq \beta$  if and only if $\beta-\alpha \in \sum_{i=1}^n \Z_{\ge 0} \alpha_i$
\item the root space $\mathfrak{g}_{\alpha}$ for a root $\alpha$
\end{itemize} 

\subsection{Some identifications and the ring $\CR$} \label{subsect:ringR}
Let $T_\RR$ be the maximal compact torus in $T$ and $\mathfrak{t}$ the Lie algebra of $T_\RR$.
Since $G$ is of rank $n$, $\mathfrak{t}$ is a real vector space of dimension $n$.
The Weyl group $W$ acts on $\mathfrak{t}$ through the differential of the conjugation map $g\to wgw^{-1}$ for $w\in W,\ g\in T$.
The $W$-action on $\mathfrak{t}$ induces the $W$-action on the dual space $\mathfrak{t}^*$ to $\mathfrak{t}$ defined by $(w(u))(v):=u(w^{-1}v)$ for $w\in W,\ u\in \mathfrak{t}^*,\ v\in\mathfrak{t}$. 

The character group $\hat T_\RR$ of $T_\RR$ determines a lattice $\mathfrak{t}^*_\Z$ through differential at the identity element of $T_\RR$.
We note that the character group $\hat T$ of $T$ is isomorphic to $\hat T_\RR$, where the isomorphism from $\hat T$ to $\hat T_\RR$ is given by restriction.
Throughout this paper we make the following identification:
\begin{equation} \label{eq:tT}
\mathfrak{t}^*_\Z=\hat T,\qquad \mathfrak{t}^*=\hat T\otimes \RR 
\end{equation}
where $\hat T$ is regarded as an additive group. We note that $\Phi$ is a subset of $\mathfrak{t}^*_\Z=\hat T$.
The Weyl group $W$ acts on $\hat T$ through conjugation, i.e., 
$w(\alpha)(g):=\alpha(w^{-1}gw)$ for $w\in W$, $\alpha\in \hat T$ and $g \in T$.
We note that the above identification preserves the $W$-actions.
We define 
\begin{equation} \label{eq:defR}
\CR:=\mbox{Sym}(\mathfrak{t}^*)=\bigoplus_{k=0}^\infty\mbox{Sym}^k(\mathfrak{t}^*),
\end{equation}
where $\mbox{Sym}(\mathfrak{t}^*)$ denotes the symmetric algebra of $\mathfrak{t}^*$ and $\mbox{Sym}^k(\mathfrak{t}^*)$ denotes the $k$-th symmetric power of $\mathfrak{t}^*$.

We shall give a different description of $\CR$ in terms of topology.
Let $ET\to BT$ be a universal principal $T$-bundle. Then $H^2(BT;\Z)$ is a free abelian group of rank $n$ and $H^*(BT;\Z)$ is a polynomial ring over $\Z$ in $n$ variables of degree 2.
Let $\CC_{\alpha}$ be the complex 1-dimensional $T$-module associated to $\alpha\in \hat T$.  
The equivariant Euler class $e^T(\CC_\alpha)$ lies in $H^2_T(pt;\Z)=H^2(BT;\Z)$.
It is known that the correspondence $\alpha\to e^T(\CC_\alpha)$ gives an isomorphism from $\hat T$ to $H^2(BT;\Z)$.
Therefore, the identification \eqref{eq:tT} is extended to 
\begin{equation} \label{eq:tTH}
\mathfrak{t}^*_\Z=\hat T=H^2(BT;\Z),\qquad \mathfrak{t}^*=\hat T\otimes \RR=H^2(BT) 
\end{equation}
and the definition \eqref{eq:defR} is to 
\begin{equation} \label{eq:defRH}
\CR:=\mbox{Sym}(\mathfrak{t}^*)=H^*(BT).
\end{equation}

\subsection{Lower ideals and Hessenberg spaces} \label{subsect:LH}
Our starting data was a semisimple linear algebraic group $G$ and its (fixed) Borel subgroup $B$. In this paper, we consider one more data, that is a lower ideal $I$ in $\Phi^+$.  
A \textbf{lower ideal} $I \subset \Phi^+$ is a collection of positive roots such that if $\alpha \in \Phi^+$ and $\beta \in I$ with $\alpha \preceq \beta$, then $\alpha \in I$. If $I$ is a lower ideal, then $H(I)=\mathfrak{b}\oplus (\bigoplus_{\alpha \in I} \mathfrak{g}_{-\alpha})$ is a $\mathfrak{b}$-submodule of $\mathfrak{g}$ containing $\mathfrak{b}$.  Conversely, one can see that any $\mathfrak{b}$-submodule of $\mathfrak{g}$ containing $\mathfrak{b}$, which is called a {\bf Hessenberg space}, is of the form $H(I)$ for some lower ideal $I$.  Therefore, the notions of lower ideals and Hessenberg spaces are equivalent.

\begin{example}\label{ex: lower ideal of A3, B3, C3}
It is convenient to visualize a lower ideal $I$.  We take types $A_3$, $B_3$, and $C_3$ to illustrate it.
We choose their simple roots as in \cite{Hum1}. We arrange their positive roots $\Phi^+_{A_3}$, $\Phi^+_{B_3}$, and $\Phi^+_{C_3}$ as follows, which is natural from the Lie-theoretical viewpoint: 
\begin{center}
\begin{picture}(430,80)
        \put(60,60){\framebox(30,20){\tiny $x_1-x_4$}}
        \put(60,40){\framebox(30,20){\tiny $x_2-x_4$}}
        \put(60,20){\framebox(30,20){\tiny $x_3-x_4$}}
        \put(30,60){\framebox(30,20){\tiny $x_1-x_3$}}
        \put(30,40){\framebox(30,20){\tiny $x_2-x_3$}}
        \put(0,60){\framebox(30,20){\tiny $x_1-x_2$}}
        
        \put(230,60){\framebox(30,20){\tiny $x_1+x_2$}}
        \put(200,60){\framebox(30,20){\tiny $x_1+x_3$}}
        \put(200,40){\framebox(30,20){\tiny $x_2+x_3$}}
        \put(170,60){\framebox(30,20){\tiny $x_1$}}
        \put(170,40){\framebox(30,20){\tiny $x_2$}}
        \put(170,20){\framebox(30,20){\tiny $x_3$}}
        \put(140,60){\framebox(30,20){\tiny $x_1-x_3$}}
        \put(140,40){\framebox(30,20){\tiny $x_2-x_3$}}
        \put(110,60){\framebox(30,20){\tiny $x_1-x_2$}}

        \put(400,60){\framebox(30,20){\tiny $2x_1$}}
        \put(370,60){\framebox(30,20){\tiny $x_1+x_2$}}
        \put(370,40){\framebox(30,20){\tiny $2x_2$}}
        \put(340,60){\framebox(30,20){\tiny $x_1+x_3$}}
        \put(340,40){\framebox(30,20){\tiny $x_2+x_3$}}
        \put(340,20){\framebox(30,20){\tiny $2x_3$}}
        \put(310,60){\framebox(30,20){\tiny $x_1-x_3$}}
        \put(310,40){\framebox(30,20){\tiny $x_2-x_3$}}
        \put(280,60){\framebox(30,20){\tiny $x_1-x_2$}}

        \put(0,0){positive roots $\Phi^+_{A_3}$}
        \put(140,0){positive roots $\Phi^+_{B_3}$}
        \put(310,0){positive roots $\Phi^+_{C_3}$}        
\end{picture}
\end{center}
\bigskip
Here in the above table, 
the elements at the left ends of rows are simple roots in each type. 
Then the partial order $\preceq$ on $\Phi^+$ can be interpreted as follows: $\alpha \preceq \beta$ if and only if $\beta$ is located northeast of $\alpha$. Thus, if $\beta$ is an element of $I$, then elements located southwest of $\beta$ must belong to $I$.  
For example, 
\[
\begin{split}
I_1&=\{x_1-x_2, x_1-x_3, x_2-x_3, x_3-x_4 \} \subset \Phi^+_{A_3},\\ 
I_2&=\{x_1-x_2, x_1-x_3, x_2-x_3, x_2, x_2+x_3, x_3 \} \subset \Phi^+_{B_3},\\ 
I_3&=\{x_1-x_2, x_1-x_3, x_1+x_3, x_2-x_3, x_2+x_3, 2x_2, 2x_3 \} \subset \Phi^+_{C_3}
\end{split}
\]
are lower ideals shown as follows:\\
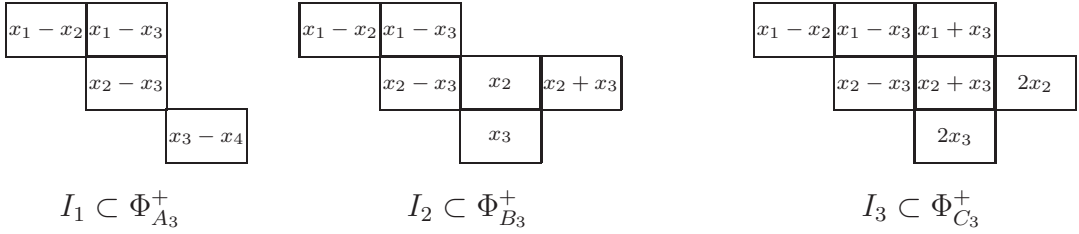
\begin{center}
\begin{picture}(430,70)
        \put(60,20){\framebox(30,20){\tiny $x_3-x_4$}}
        \put(30,60){\framebox(30,20){\tiny $x_1-x_3$}}
        \put(30,40){\framebox(30,20){\tiny $x_2-x_3$}}
        \put(0,60){\framebox(30,20){\tiny $x_1-x_2$}}
        
        \put(200,40){\framebox(30,20){\tiny $x_2+x_3$}}
        \put(170,40){\framebox(30,20){\tiny $x_2$}}
        \put(170,20){\framebox(30,20){\tiny $x_3$}}
        \put(140,60){\framebox(30,20){\tiny $x_1-x_3$}}
        \put(140,40){\framebox(30,20){\tiny $x_2-x_3$}}
        \put(110,60){\framebox(30,20){\tiny $x_1-x_2$}}

        \put(370,40){\framebox(30,20){\tiny $2x_2$}}
        \put(340,60){\framebox(30,20){\tiny $x_1+x_3$}}
        \put(340,40){\framebox(30,20){\tiny $x_2+x_3$}}
        \put(340,20){\framebox(30,20){\tiny $2x_3$}}
        \put(310,60){\framebox(30,20){\tiny $x_1-x_3$}}
        \put(310,40){\framebox(30,20){\tiny $x_2-x_3$}}
        \put(280,60){\framebox(30,20){\tiny $x_1-x_2$}}

        \put(20,0){$I_1 \subset \Phi^+_{A_3}$}
        \put(150,0){$I_2 \subset \Phi^+_{B_3}$}
        \put(320,0){$I_3 \subset \Phi^+_{C_3}$}        
\end{picture}
\end{center}
\end{example}

\bigskip

\subsection{Hessenberg varieties}
For $X\in \mathfrak{g}$ and a lower ideal $I$, the \textbf{Hessenberg variety} $\Hess(X,I)$ is defined by 
\[ \Hess(X,I):=\{g B \in G/B \mid \mbox{Ad}(g^{-1})(X) \in H(I)\},\]
where $H(I)$ is the Hessenberg space defined in \S\ref{subsect:LH}. 
An element $X\in \mathfrak{g}$ is {nilpotent} if $\mbox{ad}(X)$ is nilpotent, i.e.,\ $\mbox{ad}(X)^k=0$ for some $k>0$.
An element $X\in \mathfrak{g}$ is {semisimple} if $\mbox{ad}(X)$ is semisimple, i.e.,\ $\mbox{ad}(X)$ is diagonalizable.
An element $X\in \mathfrak{g}$ is {regular} if its $G$-orbit of the adjoint action has the largest possible dimension (cf.\ \cite{Hum1,Hum2}).

\begin{rem}\label{rem:regnilpsemi}
An element $N\in \mathfrak{g}$ is regular nilpotent if and only if $N$ is a nilpotent element of the regular (or principal) nilpotent orbit which is a unique maximal nilpotent orbit. 
Fix a basis $E_\alpha$ for each root space $\mathfrak{g}_\alpha$. 
Since a nilpotent element of the form $\sum_{\alpha_i \in \Delta }E_{\alpha_i}$ is regular,
an element $N\in \mathfrak{g}$ is regular nilpotent if and only if $N$ belongs to the adjoint orbit of the regular nilpotent element of the form
$\sum_{\alpha_i \in \Delta }E_{\alpha_i}$. 
An element $S\in \mathfrak{g}$ is regular semisimple if and only if the centralizer of $S$ in $G$ is the maximal torus $T$, i.e., $C_G(S):=\{g\in G \mid \mbox{Ad}(g)(S)=S \}=T$. 
\end{rem}

The Hessenberg variety $\Hess(X,I)$ is called \textbf{regular nilpotent} (resp.\ \textbf{regular semisimple}) if $X$ is regular nilpotent (resp.\ regular semisimple).
If $X$ and $X'$ belong to the same adjoint orbit, then $\Hess(X,I)$ and $\Hess(X',I)$ are isomorphic.
From this fact together with Remark~\ref{rem:regnilpsemi}, we may assume that $N$ in the regular nilpotent Hessenberg variety $\Hess(N,I)$ is of the form $\sum_{\alpha_i \in \Delta }E_{\alpha_i}$. 

The following is a summary of some results from \cite{Pr1,Pr} about $\Hess(N,I)$ (\cite{Ty} in the classical types). 

\begin{theorem}[\cite{Pr1,Pr}]
\label{theorem:regnilp}
The regular nilpotent Hessenberg variety $\Hess(N,I)$ has no odd degree cohomology and is of dimension equal to $|I|$. 
Moreover, the Poincar\'e polynomial of $\Hess(N,I)$ is given by
\[
\Poin(\Hess(N,I), \sqrt{\s})=\sum_{w\in W \atop w^{-1}(\Delta)\subset  (-I) \cup \Phi^+} \s^{|N(w)\cap I|}
\]
where $N(w)=\{\alpha\in\Phi^+\mid w(\alpha)\prec 0\}$, 
and is palindromic, i.e., 
\[
\s^{|I|}\Poin(\Hess(N,I), {\sqrt \s}^{-1})=\Poin(\Hess(N,I), \sqrt{\s}).
\]
\end{theorem}

\begin{proof}
We shall briefly explain how the theorem follows from results in \cite{Pr1} and \cite{Pr}.  
Let $X_w=BwB/B$ be the Schubert cell of the flag variety $G/B$ associated to $w\in W$.
Then the intersections $X_w \cap \Hess(N,I)$ $(w\in W)$ form a complex affine paving of $\Hess(N,I)$ (\cite[Theorem 4.10]{Pr1}), so that $\Hess(N,I)$ has no odd degree cohomology (see \cite[\S B.3, Lemma 6]{Fult} for paving).  By \cite[Corollary 4.13]{Pr1}, we have 
\[
\begin{split}
&\text{$X_w \cap \Hess(N,I)\not=\emptyset$ if and only if $w^{-1}(\Delta)\subset (-I) \cup \Phi^+$, and then}\\
&\text{ $\dim_{\CC} X_w \cap \Hess(N,I)=|N(w^{-1})\cap w(-I)|=|N(w)\cap I|$.}
\end{split}
\]
It follows that  
\[
\Poin(\Hess(N,I), \sqrt{\s})
=\sum_{w\in W \atop w^{-1}(\Delta)\subset (-I) \cup \Phi^+} \s^{|N(w)\cap I|}.
\]
One can see $\dim_\CC\Hess(N,I)=|I|$ from the above formula but it is also a special case of \cite[Corollary 2.7]{Pr}.  The palindromicity of the Poincar\'e polynomial of $\Hess(N,I)$ is \cite[Proposition 3.4]{Pr}. 
\end{proof}

The following is a summary of results from \cite{dMPS} about $\Hess(S,I)$, where $S\in\mathfrak{g}$ is regular semisimple.  Since $S$ is semisimple, the action of the maximal torus $T$ on $G/B$ leaves $\Hess(S,I)$ invariant.  

\begin{theorem}[{\cite[Theorems 6, 8 and Lemma 7]{dMPS}}] 
\label{theorem:regsemi}
The regular semisimple Hessenberg variety $\Hess(S,I)$ is smooth equidimensional of dimension equal to $|I|$ and has no odd-dimensional cohomology.
The $T$-fixed point set $\Hess(S,I)^T$ of $\Hess(S,I)$ agrees with that of $G/B$, so $\Hess(S,I)^T$ can be identified with the Weyl group $W$.  The Poincar\'e polynomial of $\Hess(S,I)$ is given by 
\[
\Poin(\Hess(S,I), \sqrt{\s})=\sum_{w\in W} \s^{|N(w)\cap I|}.
\]
Finally, the tangent space $T_w\Hess(S,I)$ of $\Hess(S,I)$ at $w\in W$ is of the form 
\begin{equation} \label{eq:TwHess}
T_w\Hess(S,I)=\bigoplus_{\alpha\in -I}\CC_{w(\alpha)} \quad \text{as $T$-modules},
\end{equation}
where $\CC_{w(\alpha)}$ denotes the complex $1$-dimensional $T$-module determined by $w(\alpha)$.  
\end{theorem}


\subsection{Ideal arrangements and ideal $\mathfrak{a}(I)$} \label{subsect:IA}
Let $I$ be a lower ideal in $\Phi^+$.
Since $\alpha\in I$ is a linear function on $\mathfrak{t}$, its kernel $\ker \alpha$ is a hyperplane in $\mathfrak{t}$.  We consider the arrangement $\A_I$ defined by  
\[\A_I:=\{\ker\alpha \mid \alpha \in I\}.\]
When $I=\Phi^+$, $\A_{\Phi^+}$ is called the \textbf{Weyl arrangement} and for a general lower ideal $I$ we call $\A_I$ the \textbf{ideal arrangement}\footnote{It is called the \emph{ideal subarrangement} of the Weyl arrangement $\A_{\Phi^+}$ in \cite{ABCHT}.} associated to $I$.

Remember that 
\begin{equation} \label{eq:DAI}
D(\A_I)=\{ \theta\in \Der\CR=\CR\otimes \mathfrak{t}\mid \theta(\alpha)\in \CR\alpha \ ({}^\forall \alpha\in I)\}.
\end{equation}


\begin{define}\label{ideala}
For a lower ideal $I$ and a $W$-invariant non-degenerate quadratic form $Q\in \mbox{Sym}^2(\mathfrak{t}^*)^W\subset \CR$ on $\mathfrak{t}$, we define 
\begin{equation} \label{eq:defaI}
\mathfrak{a}(I):=\{ \theta(Q) \mid \theta\in D(\A_I)\}.
\end{equation}
\end{define}
Since $D(\A_I)$ is an $\CR$-module, $\mathfrak{a}(I)$ is an ideal of $\CR$.
This ideal plays an important role in our argument.
If we choose a $W$-invariant inner product on $\mathfrak{t}$ and $x_1,\dots,x_n$ is an orthonormal linear system,
then we may take $Q=\sum_{i=1}^nx_i^2$, that is nothing but the chosen $W$-invariant inner product.  

\begin{rem}\label{ideala2}
The $W$-invariant non-degenerate quadratic form $Q$ on $\mathfrak{t}$ is not unique, hence it is not clear whether the ideal $\mathfrak{a}(I)$ is independent of the choice of $Q$. However,  it is independent of the choice by the following reason.
Suppose that the Lie algebra $\mathfrak{g}$ of $G$ is simple, which is equivalent to $\mathfrak{t}$ being irreducible as a $W$-module.
Then  $\mbox{Sym}^2(\mathfrak{t}^*)^W$ is of real dimension one\footnote{
An element in $\mbox{Sym}^2(\mathfrak{t}^*)^W$ determines a $W$-equivariant linear map $\mathfrak{t}\to \mathfrak{t}^*$,
so choosing a $W$-invariant inner product on $\mathfrak{t}$, we may think of $\mbox{Sym}^2(\mathfrak{t}^*)$ as the algebra $\mbox{End}(\mathfrak{t})^W$ of $W$-equivariant endomorphisms of $\mathfrak{t}$.
Since $\mathfrak{t}$ is irreducible as a $W$-module, $\mbox{End}(\mathfrak{t})^W$ is isomorphic to $\R, \CC$ or  the quaternion filed as $\R$-algebras and $\mbox{End}(\mathfrak{t})^W$ is isomorphic to $\R$ if and only if $\mathfrak{t}$ does not admit a complex structure as a $W$-module.}
because the irreducible $W$-module $\mathfrak{t}$ does not admit a  complex structure as is well-known (or easily checked).
Since $Q$ is an element of $\mbox{Sym}^2(\mathfrak{t}^*)^W$, this means that $Q$ is unique up to a nonzero scalar multiple.  
Therefore, $\mathfrak{a}(I)$ is independent of the choice of $Q$ in this case.
Suppose that the semisimple Lie algebra $\mathfrak{g}$ decomposes into a direct sum of simple Lie algebras.
Then the Lie algebra $\mathfrak{t}$, the ideal arrangement $\A_I$, and  a $W$-invariant quadratic form $Q$ on $\mathfrak{t}$ decompose accordingly.
This together with the observation of the simple case implies the independence of the choice of $Q$ for $\mathfrak{a}(I)$.    
\end{rem}

Because of the independence discussed in Remark \ref{ideala2}, 
we may take a $W$-invariant inner product on $\mathfrak{t}$ as $Q$.
Then the inner product determines an isomorphism $\mathfrak{t}\to \mathfrak{t}^*$ as $W$-modules and one can see that $\mathfrak{a}(I)$ agrees with the image of $D(\A_I)$ by the following $\CR$-module map: 
\begin{equation} \label{eq:DerRR} 
\Der\CR=\CR\otimes\mathfrak{t} \to \CR\otimes\mathfrak{t}^*\to \CR
\end{equation}
where the second map is the multiplication map (note that $\mathfrak{t}^*$ is the degree one piece of $\CR$).  

The following theorem plays a key role when we consider the ideal arrangements $\A_I$. 

\begin{theorem}[Ideal-free theorem, {\cite[Theorem 1.1]{ABCHT}}]\label{Ideal-free theorem}
Let $I \subset \Phi^+$ be a lower ideal. Then $\A_{I}$ is free. Moreover, $\exp(\A_I)$ 
coincides with the 
dual partition of the height distribution in $I$.
\label{idealfree}
\end{theorem}

Here,
the \textbf{height distribution} in $I$ is a sequence $(n,i_1,i_2,\dots,i_m)$, where $i_j$ is the number of height $j$ positive roots and $m$ is the maximum of the height of positive roots in $I$.
Also, for a non-increasing sequence $(i_0,i_1,\ldots,i_m)$ of non-negative integers, its 
dual partition is given by $((0)^{i_0-i_1},(1)^{i_1-i_2},\ldots,
(m-1)^{i_{m-1}-i_{m}},(m)^{i_m})$, where $i_0 = n$ and $(i)^j$ denotes the $j$-copies of $i$.

\begin{define}
For a lower ideal $I \subset \Phi^+$, we denote by $(d_1^I,\ldots,d_n^I)$ the dual partition of the height 
distribution of the positive roots in $I$. 
\label{DPHD}
\end{define}

By Theorem \ref{idealfree}, it is clear that $\exp(\A_I)=(d_1^I,\ldots,d_n^I)$. 
The following theorem provides a starting point in our argument.

\begin{theorem}
$\mathfrak{a}(\Phi^+) =(\CR^W_+)$, where the right hand side denotes the ideal in $\CR$ generated by the $W$-invariants $\CR_+^W=\bigoplus_{k>0} \CR^W_k$.
\label{coinv}
\end{theorem}

\begin{proof}
This result follows from a result of K. Saito (\cite{S}, see also 
\cite{S3} and \cite{T3}) but we shall give a short proof for the sake of the reader's convenience. 

The exterior derivative $df$ of $f\in \CR$ lies in $\CR\otimes\mathfrak{t}^*$.
On the other hand, $\Der\CR$ can be thought of as $\CR\otimes\mathfrak{t}$ as explained above.  We choose a $W$-invariant inner product on $\mathfrak{t}$.
It induces a $W$-equivariant $\CR$-module isomorphism between $\CR\otimes\mathfrak{t}^*$ and $\Der\CR=\CR\otimes\mathfrak{t}$.
Then, through the map \eqref{eq:DerRR}, the element $\nabla_f$ in $\Der\CR$ corresponding to $df \in \CR\otimes\mathfrak{t}^*$, that is the gradient of $f$, maps to $(\deg f)f$ when $f$ is homogeneous.
Indeed, if we choose an orthonormal linear coordinate system $x_1,\dots,x_n$ on $\mathfrak{t}$, then ${\nabla_f=\sum_{i=1}^n\frac{\partial f}{\partial x_i}\frac{\partial }{\partial x_i}}$
because ${df=\sum_{i=1}^n\frac{\partial f}{\partial x_i}dx_i}$, so $\nabla_f$ maps to ${\sum_{i=1}^n\frac{\partial f}{\partial x_i}x_i=(\deg f)f}$ through \eqref{eq:DerRR}. 

By the discussion in the previous paragraph,
what we must prove is that $\{\nabla_f:f \in \CR^W_+\}$ generates $D(\A_{\Phi^+})$.
Since the correspondence $f\to \nabla_f$ is $W$-equivariant, for any $f \in \CR^W_+$ and $\alpha\in \Phi^+$, we have 
$$\nabla_f(\alpha)=\nabla_{s_\alpha f}(\alpha)=s_\alpha(\nabla_f(s_{\alpha}^{-1}(\alpha)))=-s_\alpha(\nabla_f(\alpha))$$
where $s_\alpha$ denotes the reflection through the hyperplane $\ker\alpha$. 
Therefore, $\nabla_f(\alpha)$ restricted to $\ker\alpha$ vanishes and this implies that $\nabla_f(\alpha)$ is divisible by $\alpha$.
Hence $\nabla_f \in D(\A_{\Phi^+})$ for any $f \in \CR^W_+$.
By Chevalley's theorem in \cite{Ch},
$\CR^W$ is generated by homogeneous $n$ polynomials $P_1,P_2,\dots,P_n$ with $\sum_{i=1}^n (\deg P_i-1)=|\Phi^+|$.
Also, it is known that the Jacobian $\det (\partial P_i/\partial x_j)_{ij}$ is non-zero (see \cite{Ste}), which implies that $\nabla_{P_1},\dots,\nabla_{P_n}$ are $\CR$-independent.
These facts and Saito's criterion (Theorem \ref{Saito}) prove that $\nabla_{P_1},\dots,\nabla_{P_n}$ is an $\CR$-basis of $D(\A_{\Phi^+})$, and hence $\mathfrak a(\Phi^+)=(\CR^W_+)$.
\end{proof}

\section{The ideal $\mathfrak{a}(I)$ and its residue algebra}

In this section we study properties of the ideal $\mathfrak{a}(I)$ and the algebra $\CR/\mathfrak{a}(I)$ for a lower ideal $I$, where 
\[
\mathfrak{a}(I)=\{\theta(Q) \mid \theta \in D(\A_I)\},
\]
see \S~\ref{subsect:IA}. Recall that $(d_1^I,\ldots,d_n^I)$ denotes the dual partition of the height distribution of positive roots in $I$ in Definition \ref{DPHD}, which coincides with $\exp(\A_I)$, see \S~\ref{subsect:HA}.

\begin{prop} [\cite{AMN}]
\label{CI}
Let $I$ be a lower ideal. Then $\CR/\mathfrak{a}(I)$ is 
a complete intersection of socle degree $|I|$ and
\[
F(\CR/\mathfrak{a}(I),\s)=\prod_{i=1}^n (1+\s+\cdots+\s^{d_i^I}).
\]
\end{prop}

\begin{proof}
This is a special case of the result in \cite{AMN},
where the statement is proved for all free arrangements.
Here we give a different proof for the 
ideal arrangement case.
We know that there is a surjection 
$\CR/\mathfrak{a}(\Phi^+) \rightarrow \CR/\mathfrak{a}(I)$ since 
$D(\A_{\Phi^+}) \subset D(\A_I)$. Also, $\dim_\R \CR/\mathfrak{a}(\Phi^+) < \infty$ since 
it is a coinvariant algebra by Theorem~\ref{coinv}. Hence $\dim_\R \CR/\mathfrak{a}(I)< \infty$. Since $\A_I$ is free by Theorem 
\ref{idealfree}, $D(\A_I)$ is generated by $n$ elements of degrees $d^I_1,\dots,d_n^I$.
Thus $\mathfrak{a}(I)$ is generated by $n$ polynomials
of degrees $d^I_1+1,\dots,d_n^I+1$
and the desired statement follows from Lemmas \ref{completeintersection} and \ref{socledegree}.
\end{proof}

The following proposition is the key to prove our main theorems.  

\begin{prop}
\label{inj}
Let $I \subsetneq \Phi^+$ be a lower ideal and $\beta_I:= \prod_{\alpha \in \Phi^+ \setminus I} \alpha$.
Then 
\[
\mathfrak{a}(I)=\mathfrak{a}(\Phi^+):\beta_I.
\]
\end{prop}

\begin{proof}
Let $\alpha \in \Phi^+ \setminus I$ be an element such that
$I':=I \cup\{\alpha\}$ is a lower ideal.
It suffices to prove that $\mathfrak{a}(I)= \mathfrak{a}(I'):\alpha.$

We first show $\alpha \not \in \mathfrak a(I')$.
Suppose contrary that $\alpha \in \mathfrak a(I')$.
Then there is $\theta \in D(\A_{I'})$ such that $\theta(Q)=\alpha$.
Since $\alpha$ is a linear form, $\theta$ has degree zero and $\theta=\nabla_\gamma$ for some linear form $\gamma \in \CR$, where $\nabla_\gamma$ denotes the gradient of $\gamma$.
Moreover $\gamma=\frac 1 2 \alpha$ since $\frac 1 2 \nabla_\gamma(Q)=\gamma$.
However, this implies that $\theta(\alpha)=\frac 1 2\nabla_\alpha(\alpha)$ is a non-zero constant, contradicting $\theta \in D(\A_{I'})$.

By definition, $\alpha \theta \in D(\A_{I'})$ for any 
$\theta \in D(\A_I)$.
This implies that $\alpha f \in \mathfrak{a}(I')$ for any $f \in \mathfrak{a}(I)$.
Since both $\CR/\mathfrak{a}(I)$ and $\CR/\mathfrak{a}(I')$
are complete intersections
and their socle degrees are $|\mathfrak{a}(I)|$ and $|\mathfrak{a}(I')|=|\mathfrak{a}(I)|+1$ respectively by 
Proposition \ref{CI}, the desired statement follows from Lemma \ref{colonideal}.
\end{proof}

\section{Regular nilpotent Hessenberg varieties} \label{sect:RegNil}

Let $G$ be a semisimple linear algebraic group and $T\subset B$ a maximal torus and a Borel subgroup of $G$ respectively as before.
To each character $\alpha\in \hat T$, one can associate a complex line bundle $L_\alpha$ over $G/B$.
Taking the Euler class $e(L_\alpha)$ of $L_\alpha$ induces a homomorphism from $\hat T$ to $H^2(G/B)$, which extends to a homomorphism 
\begin{equation} \label{eq:psi}
\varphi\colon \CR=\mbox{Sym}(\hat T\otimes\R)\to H^*(G/B).
\end{equation}
This map doubles the grading on $\CR$ and is surjective by Borel's theorem.  Composing $\varphi$ with the restriction map $H^*(G/B)\to H^*(\Hess(N,I))$, we obtain a homomorphism 
\begin{equation} \label{eq:psiI}
\varphi_I\colon \CR\to H^*(\Hess(N,I)).
\end{equation}
In this section we introduce an ideal $\mathfrak{n}(I)$ in $\CR$ associated to $I$, which is contained in the kernel of $\varphi_I$, and 
show that $\mathfrak{n}(I)$'s have similar properties 
to the ideals $\mathfrak{a}(I)$'s. 
In order to define and study $\mathfrak{n}(I)$ we will use equivariant cohomology.   

\subsection{$T$-action on $G/B$}
We begin with the study of the $T$-action on $G/B$.  As is well-known, the identity map on $T$ extends to a homomorphism from $B$ to $T$,
so any character $\alpha$ of $T$ extends to a character $\tilde \alpha$ of $B$.  We associate a complex line bundle over $G/B$ defined by
\[
L_\alpha:=G\times_B\CC_{\tilde\alpha}\to G/B
\]
where $\CC_{\tilde\alpha}$ is the complex 1-dimensional $B$-module via $\tilde\alpha$ and  $G\times_B \CC_{\tilde\alpha}$ is the quotient space of $G\times \CC_{\tilde\alpha}$ by the $B$-action given by $b(g,z)=(gb^{-1},\tilde\alpha(b)z)$ for $(g,z)\in G\times \CC_{\tilde\alpha}$ and $b\in B$.
The left $T$-action on $G$ makes $L_\alpha$ a $T$-equivariant complex line bundle.
Then the assignment $\hat{T} \ni \alpha \mapsto e^T(L_\alpha)\in H^2_T(G/B)$ induces a homomorphism 
\begin{equation} \label{eq:ringGB}
\CR\to H^*_T(G/B)
\end{equation} 
which doubles the grading on $\CR$ and agrees with the map $\varphi$ in \eqref{eq:psi} when composed with the restriction map $H^*_T(G/B)\to H^*(G/B)$.  

The fixed point set $(G/B)^T$ of the $T$-action on $G/B$ is given by 
\begin{equation} \label{eq:GBW}
(G/B)^T=\bigsqcup_{w\in W} wB.
\end{equation}
We identify $(G/B)^T$ with $W$ through the correspondence $wB\to w$.
We denote the image of $f\in H^*_T(G/B)$ under the restriction map $H^*_T(G/B)\rightarrow H^*_T(w)=H^*(BT)$ for $w\in W$ by $f|_w$. 

As explained in \S\ref{subsect:ringR}, we identify $\hat T$ with $H^2(BT;\Z)$.
Remember that the Weyl group $W$ acts on $\hat T$ by conjugation,
i.e.,\ $w(\alpha)(g)=\alpha(w^{-1}gw)$ for $w\in W$, $\alpha\in \hat T$, and $g\in T$. 

\begin{rem}
The automorphism of $T$ defined by $g\to w^{-1}gw$ for $g\in T$ and $w\in W$ induces a self-homeomorphism of $BT$ and an automorphism of $H^*(BT)$.
Therefore we obtain another $W$-action on $H^2(BT)$ but this agrees with the $W$-action introduced above through $\hat T$.  
\end{rem}

\begin{lemma}\label{lemma:LocFlag}
Let $\alpha$ be a character of $T$.  Then 
$e^T(L_\alpha )|_w=w(\alpha)$ for any $w\in W$. 
\end{lemma}

\begin{proof}
Let $L_\alpha |_w$ be the fiber of the line bundle $L_\alpha$ at $w\in W$.
Since $w$ is a $T$-fixed point, $L_\alpha|_w$ is a $T$-module and 
\begin{equation}\label{eq:PropetyChernClass}
e^T(L_\alpha )|_w=e^T(L_\alpha |_w)
\end{equation}
from the naturality of Euler class.  Since $L_\alpha |_w=\{[w,x]\mid x\in \CC_\alpha \}$,
it follows from the definition of $L_\alpha$ and the $T$-action on $L_\alpha$ that 
\[
g[w,x]=[gw,x]=[w(w^{-1}gw),x]=[w,\alpha (w^{-1}gw)x]=[w,(w(\alpha ))(g)x]\ \ \text{for $g\in T$}.
\]
This shows that the $T$-module $L_\alpha|_w$ is given by the character $w(\alpha)$ and hence 
\begin{equation}\label{eq:T-acionFiber}
e^T(L_\alpha |_w)=w(\alpha).
\end{equation}
The identities \eqref{eq:PropetyChernClass} and \eqref{eq:T-acionFiber} prove the lemma.
\end{proof}

\subsection{$S$-action on $\Hess(N,I)$} \label{subsect:SA}
The $T$-action on the flag variety $G/B$ does not leave a regular nilpotent Hessenberg variety $\Hess(N,I)$ invariant in general.
However, there is a $\CC^*$-subgroup $S$ of $T$ which leaves $\Hess(N,I)$ invariant (\cite{HT}) and it plays an important role in \cite{AHHM} to study $\Hess(N,I)$ in type $A$.
The definition of the subgroup $S$ is as follows.  
We consider a homomorphism 
\[
\text{$\prod_{i=1}^n \alpha_i: T\rightarrow (\CC^*)^n$ given by $g\mapsto (\alpha_1(g),\dots,\alpha_n(g))$,}
\] 
where $\alpha_1,\dots,\alpha_n$ are simple roots of $G$.
Then the $\CC^*$-subgroup $S$ is the identity component of the preimage of the diagonal subgroup $\{(c,\cdots,c) \mid c\in \CC^* \}$ of $(\CC^*)^n$. 

\begin{prop}[{\cite[Lemma~5.1 and Proposition~5.2]{HT}}] \label{prop:FixedPoint}
The $S$-fixed point set $\Hess(N,I)^S$ is the intersection $\Hess(N,I)\cap (G/B)^T$ and under the natural identification $(G/B)^T=W$ by \eqref{eq:GBW}, we have
\[
\Hess(N,I)^S=\{w\in W \mid w^{-1}(\Delta)\subset (-I)\cup \Phi^+ \}.
\]
\end{prop}

By abuse of notation, we use the same symbol $e(L_\alpha)$ for the image of $e(L_\alpha)$ under the restriction map $H^*(G/B)\rightarrow H^*(\Hess(N,I))$. 
Similarly, we use the symbol $e^S(L_\alpha)$ for the image of $e^T(L_\alpha)$ under the restriction map $H^*_T(G/B)\rightarrow H^*_S(\Hess(N,I))$.
Let $t$ be the character of $S$ obtained as the composition of the inclusion $\prod \alpha_i:S\hookrightarrow T$ and the isomorphism from the diagonal subgroup $\{(c,\cdots,c) \mid c\in \CC^* \}$ to $\CC^*$ given by $(c,\cdots,c)\mapsto c$.
Then, through the equivariant Euler class, we have identification   
\[
H^*(BS)=\R[t].
\]
The inclusion map $\iota\colon S\hookrightarrow T$ induces a homomorphism $\iota^*: H^*(BT)\rightarrow H^*(BS)$ and it follows from the definition of $S$ that 
\begin{equation}\label{eq:STHom}
\iota^*(\alpha_i)=t \quad\text{for all $i=1,\cdots,n$.}
\end{equation}

\begin{prop} \label{prop:Svanish}
Let $I \subsetneq \Phi^+$ be a lower ideal and $\alpha \in \Phi^+ \setminus I$ an element such that $I':=I \cup\{\alpha\}$ is also a lower ideal. 
Then 
\[\text{$e^S(L_\alpha)|_v=-t$\quad for $v \in \Hess(N,I')^S \setminus \Hess(N,I)^S$.}\]
\end{prop}

\begin{proof}
By Proposition~\ref{prop:FixedPoint}, $w\in \Hess(N,I)^S$ if and only if $w^{-1}(\Delta)\subset (-I)\cup \Phi^+$.
Therefore, for $v \in \Hess(N,I')^S \setminus \Hess(N,I)^S$, there exists $\alpha_i\in\Delta$ such that 
\begin{equation}\label{eq:FP}
v(\alpha)=-\alpha_i.
\end{equation}
Using the following commutative diagram
\begin{equation*}
\begin{CD}
H^{\ast}_{T}(G/B)\ \ \ \ \ @>
>> \ \ \ \ \ \ \ \ H^{\ast}_{T}(v)=H^*(BT)\\
@V{}VV @VV{\iota^*}V\\
H^{\ast}_{S}(\Hess(N,I')) \ \ \ \ \ @> 
>> H^{\ast}_{S}(v)=H^*(BS)=\R[t]
\end{CD}
\end{equation*}
where all the homomorphisms are induced from the inclusion maps, we have
\begin{align*}
e^S(L_\alpha )|_v=& \iota^* (e^T(L_\alpha )|_v)\\
=&\iota^*(v (\alpha) ) \hspace{50pt} \mbox{by Lemma \ref{lemma:LocFlag}}\\
=& \iota^*(-\alpha_i ) \hspace{52pt} \mbox{by \eqref{eq:FP}}\\
=& -t \hspace{72pt} \mbox{by \eqref{eq:STHom}},
\end{align*}
proving the proposition. 
\end{proof}

\subsection{Ideal $\mathfrak{n}(I)$} 
Recall the homomorphism $\CR\to H_T^*(G/B)$ in \eqref{eq:ringGB}.  Composing this with the restriction homomorphism $H^*_T(G/B)\to H^*_S(\Hess(N,I))$, we obtain a homomorphism 
\[
\phi_I\colon \CR\to H^*_S(\Hess(N,I))
\]
sending  $\alpha\in \hat T\subset\CR$ to $e^S(L_\alpha)$.  The map $\phi_I$ composed with the restriction map $H^*_S(\Hess(N,I))\to H^*(\Hess(N,I))$ is the map 
\begin{equation*} \label{eq:psiI2}
\varphi_I \colon \CR\to H^*(\Hess(N,I))
\end{equation*}
in \eqref{eq:psiI}, which sends $\alpha\in\hat T$ to $e(L_\alpha)\in H^2(\Hess(N,I))$. 

Through the projection $ES\times_S \Hess(N,I)\to ES/S=BS$, one can regard an element of $H^*(BS)=\R[t]$ as an element of $H^*_S(\Hess(N,I))$.  Therefore, the homomorphism $\phi_I$ naturally extends to a homomorphism 
\begin{equation} \label{eq:varphiI}
\varphi_I^S\colon \CR[t]\to H^*_S(\Hess(N,I))
\end{equation}
sending $t$ to $t$.  We define 
\begin{equation} \label{eq:hatJ}
\begin{split}
\mathfrak{n}_S(I):&=\{f(t) \in \CR[t] \mid \varphi_I^S(f(t))=0\},\\ 
\mathfrak{n}(I):&=\{ f(0)\in \CR\mid f(t)\in \mathfrak{n}_S(I)\}.
\end{split}
\end{equation}
We note that $\mathfrak{n}(I)$ is contained in the kernel of $\varphi_I$, which follows from the following commutative diagram:  
\begin{equation} \label{eq:psivarphi}
  \xymatrix{
   \CR[t] \ar[rr]^-{\varphi_I^S} \ar[d] & & H^*_S(\Hess(N,I)) \ar[d]\\ 
   \CR \ar[rr]_-{\varphi_I} \ar[urr]^-{\phi_I} & & H^*(\Hess(N,I))
  }
\end{equation}
where the left vertical map is the evaluation at $t=0$ and the right one is the restriction map. 

\begin{lemma} \label{lemm:kerpsiI}
If $\varphi_I$ is surjective, then $\mathfrak{n}(I)$ agrees with the kernel $\ker\varphi_I$ of $\varphi_I$. 
\end{lemma}

\begin{proof}
Since we know $\mathfrak{n}(I)\subset \ker\varphi_I$, it suffices to prove $\mathfrak{n}(I)\supset \ker\varphi_I$ when $\varphi_I$ is surjective.
We note that since $H^*_S(\Hess(N,I))=H^*(\Hess(N,I))\otimes H^*(BS)$, the surjectivity of $\varphi_I$ implies the surjectivity of $\varphi_I^S$.
Let $h\in \ker\varphi_I$.
Since $H^*(BS)=\R[t]$ and $\varphi_I(h)=0$, $\varphi_I^S(h)$ is divisible by $t$, i.e., $\varphi_I^S(h)$ is of the form $t\tilde f$ with $\tilde f\in H^*_S(\Hess(N,I))$, where $h$ is regarded as an element of $\CR[t]$ in a natural way.
Since $\varphi_I^S$ is surjective, there exists an element $\tilde h\in \CR[t]$ such that $\varphi_I^S(\tilde h)=\tilde f$.
Then $\varphi_I^S(h-t\tilde h)=0$ while $h-t\tilde h$ maps to $h$ by the evaluation map at $t=0$.
This shows that $h$ is in $\mathfrak{n}(I)$.  
\end{proof}


\begin{cor} \label{coro:nPhi+}
$\mathfrak{n}(\Phi^+)=(\CR^W_+)$.  
\label{hatcoinvariant}
\end{cor}

\begin{proof}
When $I=\Phi^+$, we have $\Hess(N,I)=G/B$ and $\varphi_I=\varphi\colon \CR\to H^*(G/B)$.
Since $\varphi$ is surjective and its kernel is $(\CR^W_+)$ by Borel's theorem, the corollary follows from Lemma~\ref{lemm:kerpsiI}.  
\end{proof}

The above corollary corresponds to Theorem~\ref{coinv} and the following lemma corresponds to Proposition~\ref{inj} in terms of $\mathfrak{a}(I)$. 

\begin{lemma} \label{lemm:GysinNil}
Let $I \subsetneq \Phi^+$ be a lower ideal and $\beta_I=\prod_{\alpha \in \Phi^+ \setminus I} \alpha$ as in Proposition~\ref{inj}. Then $\mathfrak{n}(I)\subset \mathfrak{n}(\Phi^+):\beta_I$. 
\end{lemma}

\begin{proof}
It suffices to show that if $\alpha$ is an element of  $\Phi^+ \setminus I$ such that $I'=I \cup\{\alpha\}$ is also a lower ideal, then $\mathfrak{n}(I)\subset \mathfrak{n}(I')\colon \alpha$, i.e., $\alpha\mathfrak{n}(I)\subset \mathfrak{n}(I')$. 

By \eqref{eq:hatJ}, any element of $\mathfrak{n}(I)$ is of the form $f(0)$ for some $f(t)\in \mathfrak{n}_S(I)$. We claim $(\alpha+t)f(t)\in \mathfrak{n}_S(I')$.  Indeed, 
\[
\text{$\varphi_{I'}^S(\alpha+t)|_v=e^S(L_\alpha)|_v+t =0$\quad for all $v\in \Hess(N,I')^S\backslash \Hess(N,I)^S$}
\]
by Proposition~\ref{prop:Svanish} while
\[
\text{$\varphi_{I'}^S(f(t))|_w=\varphi_I^S(f(t))|_w=0$\quad for all $w\in \Hess(N,I)^S$}
\] 
since $f(t)\in \mathfrak{n}_S(I)$. Therefore $\varphi_{I'}^S((\alpha+t)f(t))|_w=0$ for all $w\in \Hess(N,I')^S$ and hence $\varphi_{I'}^S((\alpha+t)f(t))=0$
because the restriction map $H^*_S(\Hess(N,I'))\to H^*_S(\Hess(N,I')^S)$ is injective since $H^{odd}(\Hess(N,I'))$ vanishes by Theorem~\ref{theorem:regnilp}.
This shows that $(\alpha+t)f(t)\in \mathfrak{n}_S(I')$ as claimed.
Therefore $\alpha f(0)$ is contained in $\mathfrak{n}(I')$.
Since $f(0)$ is an arbitrary element of $\mathfrak{n}(I)$, this proves the lemma. 
\end{proof}

\section{Weyl type subsets of ideals} \label{sect:WI}

In this section, we discuss a relation between chambers of $\A_I$, $S$-fixed points of $\Hess(N,I)$, and Weyl type subsets of $I$ defined by Sommers and Tymoczko \cite{STy}.

Let $I \subset \Phi^+$ be a lower ideal.  
A subset $\Y \subset I$ is said to be of \textbf{Weyl type} if $\alpha,\beta \in \Y$ and $\alpha + \beta \in I$, then $\alpha+\beta \in \Y$,
and if $\gamma,\delta \in I \setminus \Y$ and $\gamma + \delta \in I$, then $\gamma+\delta \in I \setminus \Y$.
Let $\mathcal{W}^I$ denote the set of the Weyl type subsets of $I$.
Sommers and Tymoczko posed the following conjecture in \cite{STy}. 

\begin{conj}[\cite{STy}]\label{conj:STy}
$\displaystyle{\sum_{\Y \in \mathcal{W}^I} \s^{|\Y|}=\prod_{i=1}^n (1+\s+\cdots+\s^{d_i^I})}$, where $d_1^I,\dots,d_n^I$ are the dual partitions of the height distributions of positive roots in $I$ (see Definition~\ref{DPHD}). 
\label{STyconj}
\end{conj}


When $I=\Phi^+$, this is a well-known fact that the Poincar\'{e} polynomial 
of the flag variety coincides with the generating function of the length of $w \in W$ when $\s$ is replaced by $\s^2$.
We will prove Conjecture \ref{STyconj} in the next section as a corollary of Theorem \ref{nilpotentmain}, and here we introduce some related results.


Weyl type subsets are closely related to $S$-fixed points of $\Hess(N,I)$.
Recall that by Proposition \ref{prop:FixedPoint} we have
\[
\Hess(N,I)^S=\{ w \in W \mid w^{-1}(\Delta) \subset (-I)\cup \Phi^+\}.
\]
It is easy to see that, for any $w \in W$, the set $N(w)\cap I$ is a Weyl type subset of $I$.
The following result was proved by Sommers and Tymoczko \cite[Proposition 6.3]{STy}.

\begin{theorem}[\cite{STy}]
\label{thm:SommersTymoczko}
Let $I$ be a lower ideal.
The map $\eta: \Hess(N,I)^S \to \mathcal{W}^I$  defined by $\eta(w)=N(w)\cap I$ is a bijection.
\end{theorem}

\begin{rem}
Sommers and Tymoczko also showed in \cite{STy} that,
for $Y \in \mathcal{W}^I$, $w:=\eta^{-1}(Y) \in \Hess(N,I)^S$ 
is the smallest element in $\{v \in W \mid N(v)\cap I=Y\}$ with respect to the Bruhat order.
\end{rem}

Recall that $C(\A_I)$ is the set of chambers of the ideal arrangement $\A_I$ (see \S\ref{subsect:HA}).  

\begin{prop}
For any lower ideal $I$,
we have $|C(\A_I)|=|\mathcal{W}^I|$.
\label{fix}
\end{prop}

\begin{proof}
For $C\in C(\A_I)$ we define 
$f(C):=\{ \alpha\in I\mid \alpha(C)<0\}$,
where $\alpha(C)<0$ means that $\alpha(x)<0$ for any point $x \in C$.
Then it is obvious that $f(C)$ is an element of $\mathcal{W}^I$.  Therefore, we obtain a map
\begin{equation} \label{eq:fbijection}
f\colon C(\A_I)\to \mathcal{W}^I.
\end{equation}
This map is injective because an element $C$ of $C(\A_I)$ is determined by the signs of the values which elements of $I$ take on $C$.  

We shall prove that $f$ is surjective.  For any $\Y \in \mathcal{W}^I$, there is an element $w \in W$ such that 
\begin{equation} \label{eq:walpha}
\Y = N(w) \cap I=\{\alpha\in I \mid w(\alpha)\prec 0\}
\end{equation}
by Theorem \ref{thm:SommersTymoczko}.
Take any point $x$ in the fundamental Weyl chamber, that is, take a point $x$ satisfying $\alpha(x)>0$ for any $\alpha \in \Phi^+$.   Then we have
\[
\alpha(w^{-1}x)=(w(\alpha))(x)\begin{cases} < 0 \quad&(\alpha\in \Y),\\
>0 \quad&(\alpha\in I\backslash \Y),\end{cases}
\]
where the inequalities follow from \eqref{eq:walpha}.  This shows that if $C$ is the element of $C(\A_I)$ which contains the point $w^{-1}x$, then $f(C)=\Y$, proving the surjectivity of $f$. 
%
\end{proof}

\begin{prop}
\label{dim}
$\Poin(\Hess (N,I),\sqrt{\s})=\sum_{\Y \in \mathcal{W}^I} \s^{|\Y|}$.  Therefore, 
$$\Poin(\Hess(N,I),1)=|\mathcal{W}^I|=|C(\A_I)|=\prod_{i=1}^n (1+d_i^I).$$
\end{prop}

\begin{proof}
The equation 
$\Poin(\Hess (N,I),\sqrt{\s})=\sum_{\Y \in \mathcal{W}^I} \s^{|\Y|}$ immediately follows from Theorems~\ref{theorem:regnilp} and \ref{thm:SommersTymoczko}.
The other equations follow from the former equation with $\s=1$ plugged, Proposition~\ref{fix}, Theorem~\ref{factorization}, and Theorem~\ref{idealfree}.
%
%
\end{proof}



\section{Proof of Theorem~\ref{nilpotentmain}}

In this section we prove Theorem~\ref{nilpotentmain} in the Introduction and deduce a few corollaries, especially  we will see that Conjecture~\ref{conj:STy} immediately follows from Theorem~\ref{nilpotentmain}. 

Remember that we have a homomorphism 
\begin{equation} \label{eq:varphiI2}
\varphi_I\colon \CR\to H^*(\Hess(N,I))
\end{equation}
sending $\alpha\in\hat T\subset \CR$ to $e(L_\alpha)\in H^2(\Hess(N,I))$, see \S\ref{sect:RegNil}.  
The following theorem implies Theorem \ref{nilpotentmain} in the Introduction. 

\begin{theorem} \label{theo:nilpotent}
The map $\varphi_I$ in \eqref{eq:varphiI2} is surjective and 
$\ker\varphi_I=\mathfrak{n}(I)=\mathfrak{a}(I).$
\end{theorem}





\begin{proof}
Since $H^{odd}(\Hess(N,I))$ vanishes by Theorem \ref{theorem:regnilp}, we have
\[
H^*_S(\Hess(N,I))\cong H^*(BS)\otimes H^*(\Hess(N,I)) \quad \text{as $H^*(BS)$-modules}
\]
where $H^*(BS)=\R[t]$.  
Moreover, it follows from the definition \eqref{eq:hatJ} of $\mathfrak{n}_S(I)$ that the map $\varphi_I^S\colon \CR[t]\to H_S^*(\Hess(N,I))$ in \eqref{eq:varphiI} induces an injective $\R[t]$-homomorphism 
\begin{equation} \label{eq:varphihat}
\hat\varphi_I^S\colon \CR[t]/\mathfrak{n}_S(I) \hookrightarrow H_S^*(\Hess (N,I)). 
\end{equation}
Hence we have
\begin{align}
\label{pfmain1}
\frac 1 {1-\s} \Poin(\Hess (N,I),\sqrt{\s})=F(H_S^*(\Hess (N,I)),\sqrt{\s}) \geq F(\CR[t]/\mathfrak{n}_S(I),\s).
\end{align}
Here, for two formal power series $F(\s)=\sum a_i \s^i$ and $G(\s)=\sum b_i \s^i$, $F(\s) \ge G(\s)$ means that $a_i \ge b_i$ for all $i$. 

Since $H_S^*(\Hess (N,I))$ is a free $\R[t]$-module and $\R[t]$ is PID, the $\R[t]$-submodule $\CR[t]/\mathfrak{n}_S(I)$ is also a free $\R[t]$-module.
Thus $t$ is a nonzero divisor of $\CR[t]/\mathfrak{n}_S(I)$ and since
\[
\CR/\mathfrak{n}(I) \cong (\CR[t]/\mathfrak{n}_S(I))/ t (\CR[t]/\mathfrak{n}_S(I)),
\]
which follows from the definition \eqref{eq:hatJ}, we have
\begin{align}
\label{pfmain2}
F(\CR/\mathfrak{n}(I),\s)
=F (\CR[t]/\mathfrak{n}_S(I),\s) -\s \ \! F (\CR[t]/\mathfrak{n}_S(I),\s).
\end{align}
It follows from \eqref{pfmain1} and \eqref{pfmain2} that  
\begin{align}
\label{pfmain4-1}
\frac 1 {1-\s} \Poin(\Hess (N,I),\sqrt{\s})
 \geq F(\CR[t]/\mathfrak{n}_S(I),\s)=\frac 1 {1-\s} F (\CR/\mathfrak{n}(I),\s).
\end{align}
On the other hand, it follows from Lemma~\ref{lemm:GysinNil}, Corollary~\ref{hatcoinvariant},  Theorem \ref{coinv}, and Proposition~\ref{inj} that 
\begin{equation} \label{eq:nIaI}
\mathfrak{n}(I) \subset \mathfrak{n}(\Phi^+):\beta_I=\mathfrak{a}(\Phi^+):\beta_I=\mathfrak{a}(I).
\end{equation}
Therefore  
\begin{align}
\label{pfmain3}
F (\CR/\mathfrak{n}(I),\s) \geq F (\CR/\mathfrak{a}(I),\s).
\end{align} 
Thus, we finally get  
\begin{align}
\label{pfmain4}
\frac 1 {1-\s} \Poin(\Hess (N,I),\sqrt{\s})\geq \frac 1 {1-\s} F (\CR/\mathfrak{a}(I),\s)
\end{align}
from \eqref{pfmain4-1} and \eqref{pfmain3}.

We claim that we actually have equality in \eqref{pfmain4}. 
We have $\dim_\CC\Hess(N,I)=|I|$ by Theorem~\ref{theorem:regnilp} and the socle degree of $\CR/\mathfrak{a}(I)$ is also $|I|$ by Proposition \ref{CI}.   
Set $m=|I|$.  Then one can write 
\[
\Poin(\Hess (N,I),\sqrt{\s})=\sum_{i=0}^m a_i\s^i,\qquad F(\CR/\mathfrak{a}(I),\s)=\sum_{i=0}^m b_i \s^i
\]
with non-negative integers $a_i,b_i$.
These Hilbert series are palindromic, that is,
\begin{align}
\label{pfmain5}
a_i=a_{m-i} \quad \text{ and }\quad b_i=b_{m-i}\quad \text{ for }0 \leq i \leq m
\end{align}
by Theorem \ref{theorem:regnilp} for the symmetry of $a_k$, and by Proposition \ref{CI} for $b_k$.
Moreover, it follows from Propositions~\ref{dim} and~\ref{CI} that   
\begin{align}
\label{pfmain6}
\sum_{i=0}^m a_i =\Poin(\Hess(N,I),1)=\prod_{i=1}^n(1+d_i^I)=F(\CR/\mathfrak{a}(I),1)=\sum_{i=0}^m b_i.
\end{align}
Since the coefficient of $\s^k$ in the formal power series $\frac 1 {1-\s} (\sum_{i=0}^m c_i \s^i)$ is $\sum_{i=0}^k c_i$ for $0 \leq k \leq m$,
the inequality \eqref{pfmain4} says
\begin{equation} \label{eq:ageb}
\sum_{i=0}^k a_i \geq \sum_{i=0}^k b_i\qquad\text{for $0 \leq k \leq m$.}
\end{equation}
On the other hand, it follows from \eqref{pfmain5}, \eqref{pfmain6}, and \eqref{eq:ageb} that we get the opposite inequality to \eqref{eq:ageb}, indeed, we can set $p=\sum_{i=0}^m a_i =\sum_{i=0}^m b_i$ by \eqref{pfmain6} and have
$$\sum_{i=0}^k a_i = \sum_{i=m-k}^m a_i = p - \!\!\sum_{i=0}^{m-k-1} a_{i} \leq p -\!\! \sum_{i=0}^{m-k-1} b_{i}= \sum_{i=m-k}^m b_i = \sum_{i=0}^k b_i\quad\text{for $0 \leq k \leq m$}.$$
Hence $\sum_{i=0}^k a_i = \sum_{i=0}^k b_i$ for all $k$, i.e.,\ $a_k=b_k$ for all $k$, and $\Poin(\Hess (N,I),\sqrt{\s})=F(\CR/\mathfrak{a}(I),\s)$ which means that equality holds in \eqref{pfmain4}.
Thus equality must hold for both \eqref{pfmain1} and \eqref{pfmain3}.

The equality in \eqref{pfmain1} implies that the map $\hat\varphi_I^S$ in \eqref{eq:varphihat} is an isomorphism and hence the map $\varphi_I^S\colon \CR[t]\to H_S^*(\Hess(N,I))$ in \eqref{eq:varphiI} is surjective,
so the map $\varphi_I$ in the theorem is also surjective because both vertical maps in the commutative diagram \eqref{eq:psivarphi} are surjective.
Therefore, $\ker\varphi_I=\mathfrak{n}(I)$ by Lemma~\ref{lemm:kerpsiI}.
The equality in \eqref{pfmain3} implies $\mathfrak{n}(I)=\mathfrak{a}(I)$ since we know $\mathfrak{n}(I)\subset \mathfrak{a}(I)$ by \eqref{eq:nIaI}. 
%
%
\end{proof}

We shall mention a few corollaries which will immediately follow from Theorem~\ref{theo:nilpotent} (i.e.,\  Theorem~\ref{nilpotentmain}).  The first one is Corollary~\ref{surjection} in the Introduction, which was announced by Dale Peterson (see \cite[Theorem 3]{BC}) but his proof is not given unfortunately. 

\begin{cor}
The restriction map $H^*(G/B)\to H^*(\Hess (N,I))$ is surjective and $H^*(\Hess(N,I))$ is a complete intersection, in particular, a Poincar\'e duality algebra.  Moreover, the Poincar\'e polynomial of $\Hess(N,I)$ is given by the product formula
\begin{equation} \label{eq:7-12}
\mbox{Poin}(\Hess(N,I),\sqrt{q})=\prod_{\alpha\in I}\frac{1-q^{ht(\alpha)+1}}{1-q^{ht(\alpha)}},
\end{equation}
where $ht(\alpha)$ denotes the sum of the coefficients of $\alpha$ over the simple roots. 
\end{cor}

\begin{proof}
The map $\varphi\colon \CR\to H^*(G/B)$ followed by the restriction map $H^*(G/B)\to H^*(\Hess(N,I))$ is the map $\varphi_I\colon \CR\to H^*(\Hess(N,I))$ by definition and $\varphi_I$ is surjective by Theorem~\ref{theo:nilpotent}. 
This shows the surjectivity of the restriction map $H^*(G/B)\to H^*(\Hess(N,I))$.  

It easily follows from the definition of the exponents $d_1^I,\dots,d_n^I$ (see Definition~\ref{DPHD}) that the right hand side of \eqref{eq:7-12} agrees with the right hand side of the identity in Proposition~\ref{CI}.  Therefore, the remaining two assertions in the corollary follow from Proposition~\ref{CI} since $H^*(\Hess(N,I))\cong \CR/\mathfrak{a}(I)$ by Theorem~\ref{nilpotentmain}.  
%
\end{proof}

The following corollary is Corollary~\ref{Weyltype} in the Introduction and answers  Conjecture \ref{STyconj} affirmatively. 

\begin{cor} \label{coro:theorem1.4}
For any lower ideal $I$,
$$ \sum_{\Y \in \mathcal{W}^I} \s^{|\Y|}=\Poin(\Hess (N,I),\sqrt{\s})=F(\CR/\mathfrak{a}(I),\s)=\prod_{i=1}^n (1+\s+\cdots+\s^{d_i^I}).$$
\end{cor}
\begin{proof}
The first identity is Proposition \ref{dim}, the second is Theorem~\ref{theo:nilpotent}, and the last is Proposition~\ref{CI}.  
\end{proof}

Corollary~\ref{coro:theorem1.4} (i.e., Corollary \ref{Weyltype}) implies an interesting application to free arrangement theory. 
Recall that the generating function of the length of $w \in W$ coincides with the 
Poincar\'{e} polynomial of the flag variety and also the generating function of the number of 
reflecting hyperplanes which separates a chamber and the fundamental chamber.
For some free arrangements, this 
formula is known to hold, i.e., for some free arrangement $\A$ with $\exp(\A)=(d_1,\ldots,d_n)$, 
there is a chamber $C_0 \in C(\A)$ such that 
\begin{equation} \label{eq:super}
\sum_{C \in C(\A)} \s^{d(C,C_0)}=\prod_{i=1}^n (1+\s+\cdots+\s^{d_i}),
\end{equation}
where $d(C,C_0)$ denotes the number of hyperplanes in $\A$ which separates $C$ and $C_0$.
The class of certain free arrangements above contains, e.g., 
a supersolvable arrangement. We say that $\A$ is supersolvable if there is a filtration 
$\A_1 \subset \A_{2} \subset \cdots \subset \A_n=\A$ such that 
$\bigcap_{H \in \A_i} H$ is of codimension $i$, and 
for any distinct $H_1, H_2 \in \A_i \setminus \A_{i-1}$, there is $L \in \A_{i-1}$ such that 
$H_1 \cap H_2 \subset L$. For details, see \cite{BEZ}.


Ideal arrangements are not necessarily supersolvable. For example, in type $D$ case, 
$\A_{\Phi^+}$ itself is not supersolvable. 
As for the ideal arrangement $\A_I$, we take $C_0$ to be the element of $C(\A_I)$ which contains the fundamental Weyl chamber in $C(\A_{\Phi^+})$.
Recall the bijection $f\colon C(\A_I)\to \mathcal{W}^I$ in \eqref{eq:fbijection} defined as $f(C)=\{\alpha\in I\mid \alpha(C)<0\}$.
Since $d(C,C_0)$ agrees with $|f(C)|$ as is easily observed, the following follows from Corollary~\ref{coro:theorem1.4}. 

\begin{cor}
The identity \eqref{eq:super} holds for ideal arrangements with $C_0\in C(\A_I)$ containing the fundamental Weyl chamber in $C(\A_{\Phi^+})$.   
\label{generating}
\end{cor}


\section{Regular semisimple Hessenberg varieties}

Unlike the regular nilpotent case in Section~\ref{sect:RegNil}, the action of the maximal torus $T$ on the flag variety $G/B$ leaves a regular semisimple Hessenberg variety $\Hess(S,I)$ invariant.
The variety $\Hess(S,I)$ is smooth projective, has finitely many $T$-fixed points, has finitely many one-dimensional $T$-orbits and has no odd degree cohomology,
so that its (equivariant) cohomology ring can be described combinatorially in terms of the associated so-called GKM graph (\cite{GKM}).
The variety $\Hess(S,I)$ may not admit an action of the Weyl group $W$ in general but the associated GKM graph always does and one can define an action of $W$ on $H^*(\Hess(S,I))$ through the GKM graph as observed by Tymoczko \cite{Ty2}.
We study the ring $H^*(\Hess(S,I))^W$ of $W$-invariants.
This ring is studied in \cite{K} when $\Hess(S,I)$ is a toric variety.   

\subsection{GKM theory} \label{subsect:GKM theory}

Let $X$ be a complex projective smooth variety with an algebraic action of a $\CC^*$-torus  $T$ which satisfies the following three conditions:
\begin{enumerate}
\item[(i)] $X$ has finitely many $T$-fixed points, 
\item[(ii)] $X$ has finitely many one-dimensional $T$-orbits, 
\item[(iii)] $X$ has no odd degree cohomology. 
\end{enumerate}
Then the restriction map 
\begin{equation}\label{eq:LocalizationandGKM}
H^*_T(X)\hookrightarrow H^*_T(X^T)= \bigoplus_{x\in X^T} H^*(BT).
\end{equation}
is injective by condition (iii) and Goresky-Kottwitz-MacPherson (\cite{GKM}) gave an explicit description of the image of the restriction map, which we shall explain.
The closure ${\bar O}$ of a one-dimensional $T$-orbit $O$ is diffeomorphic to $\CC P^1$ and $\bar{O} \setminus O$ consists of exactly two $T$-fixed points, denoted by $n_O$ and $s_O$ and called the {north and south poles} of the orbit $O$.
If the weight of the $T$-action on the tangent space of $\bar O$ at the point $n_O$ is $\alpha$, then the $T$-weight on the tangent space at the point $s_O$ is $-\alpha$.
Therefore, the $T$-weight at the fixed points in the closure $\bar{O}$ is determined up to sign.
We call it the {$T$-weight} on the orbit $O$.
We regard the weight $\alpha$ as an element of $H^2(BT)$ as before.  

\begin{theorem}[\cite{GKM}]\label{theorem:GKMcondX}
The image of the restriction map in \eqref{eq:LocalizationandGKM} is given by 
\begin{align*}
\left\{(f_x) \in \bigoplus_{x\in X^T} H^*(BT) \left |  
\begin{array}{ll}
f_{n_O}-f_{s_O}\in (\alpha) \ \text{\rm for each one-dimensional $T$-orbit $O$} \\
\mbox{\rm with poles $n_O$ and $s_O$ and $T$-weight $\alpha$}\end{array}
\right\}.
\right.
\end{align*}
\end{theorem}

In this paper we call the condition in Theorem~\ref{theorem:GKMcondX} the \textbf{GKM condition} for $X$.
The GKM condition for $X$ can be visualized by a graph, called a GKM graph.
The {\bf GKM graph} for $X$ is the graph with vertices corresponding to the $T$-fixed points and edges corresponding to one-dimensional $T$-orbits.
Additionally, we equip each edge with the $T$-weight of the corresponding one-dimensional $T$-orbit, see Example~\ref{exam:GKM graph} below.  

\subsection{GKM condition for $\Hess(S,I)$}

We return to our previous setting.
The regular semisimple Hessenberg variety $\Hess(S,I)$ is smooth projective and satisfies the conditions (i), (ii), (iii) in \S\ref{subsect:GKM theory} 
(see Theorem~\ref{theorem:regsemi} and \cite{dMPS} for more details).
In fact, the $T$-fixed point set $\Hess(S,I)^T$ agrees with $(G/B)^T=\sqcup_{w\in W}wB$ and we make the natural identification 
\[
\Hess(S,I)^T=W
\]
as before throughout this section.
Since $H^{odd}(\Hess(S,I))$ vanishes, the restriction map 
\begin{equation} \label{eq:HessSIRes}
H^*_T(\Hess(S,I))\hookrightarrow H^*_T(\Hess(S,I)^T)=\bigoplus_{w\in W}H^*(BT)
\end{equation}
is injective.  
In this subsection we analyze the GKM condition for $\Hess(S,I)$ and prove the following. 

\begin{prop}[GKM condition for $\Hess(S,I)$]\label{prop:GKM}
\begin{align*}
&H^*_T(\Hess(S,I))\\
\cong& \{ (f_w)_{w\in W} \in \bigoplus _{w\in W}H^*(BT) \mid f_w-f_v \in (w(\alpha)) \ {\rm if} \ v=w s_{\alpha} \ {\rm for \ some} \ \alpha \in I \}
\end{align*}
where $s_\alpha$ is the reflection corresponding to $\alpha$. 
\end{prop}

Proposition~\ref{prop:GKM}  is known in type $A$ (\cite{Tef,Ty2}) and the proof in other 
types is essentially same as type $A$.
We shall give a proof for the reader's convenience.
For each root $\alpha$ there exists a morphism of algebraic groups $u_{\alpha}: \CC\rightarrow G$, which induces 
an isomorphism onto $u_{\alpha}(\CC)$ such that $gu_{\alpha}(c)g^{-1}=u_{\alpha}(\alpha(g)c)$ for all $g\in T$ and $c\in \CC$.
The root subgroup $U_{\alpha}$ is defined by the image of $u_\alpha$ (cf.\ \cite[Theorem 8.17 and Definition 8.18]{MT}). 

\begin{prop}{\rm (\cite{C}, \cite[Proposition 4.6]{Ty3}\rm ).}
\label{prop:1-dim orbit}
Under the identification $(G/B)^T=W$,  there is a  one-dimensional $T$-orbit with poles $w$ and $v$ $(w, v\in W)$ if and only if $v=s_\alpha w$ for some $\alpha \in \Phi^+$.
If $v=s_\alpha w$, then the one-dimensional $T$-orbit is given by $U_\alpha w\cup U_\alpha v$ and the $T$-weight on the orbit is $\alpha$.
\end{prop}

\begin{rem}
If $w^{-1}(\alpha)$ is a positive root, 
then $U_\alpha w$ is a one point set in $G/B$.
This is because $w^{-1} U_{\alpha} w =U_{w^{-1}(\alpha)} \subset B$
for such $w$ and $\alpha$.
\end{rem}

Proposition \ref{prop:1-dim orbit} together with Theorem~\ref{theorem:GKMcondX} shows that the GKM condition for $G/B$ is given by
\begin{align*}
\{ (f_w)_{w\in W} \in \bigoplus _{w\in W}H^*(BT) \mid f_w-f_v \in (\alpha) \ {\rm if} \ v=s_{\alpha} w  \ {\rm for \ some} \ \alpha \in \Phi^+ \}.
\end{align*}
Using the equality $w^{-1}s_{\alpha}w=s_{w^{-1}(\alpha)}$, we can rewrite the above GKM condition for $G/B$ as follows:
\begin{align*}
\{ (f_w)_{w\in W} \in \bigoplus _{w\in W}H^*(BT) \mid f_w-f_v \in (w(\alpha)) \ {\rm if} \ v=w s_{\alpha} \ {\rm for \ some} \ \alpha \in \Phi^+ \}.
\end{align*}
This description is exactly the right hand side for $I=\Phi^+$ in Proposition~\ref{prop:GKM}.

\begin{example} \label{exam:GKM graph}
As mentioned in \S\ref{subsect:GKM theory}, a GKM condition can be visualized in terms of a GKM graph.
For example, the GKM graph associated to the flag variety $G/B$ of type $A_2$ is the following labeled graph (cf. \cite{Ty2}): \\
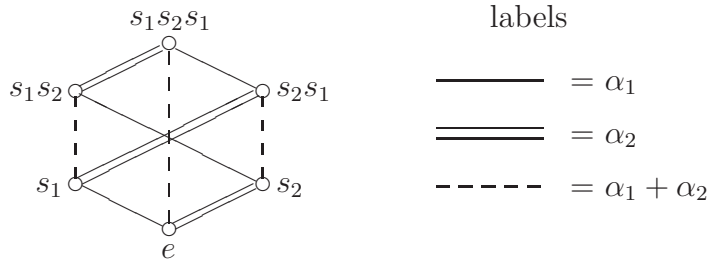
\begin{center}
\begin{picture}(300,90)
        \put(50,14){\circle{5}}
        \put(50,84){\circle{5}}
        \put(85,31){\circle{5}}
        \put(85,66){\circle{5}}
        \put(15,31){\circle{5}}
        \put(15,66){\circle{5}}

        \put(82.5,32){\line(-2,1){66}}
        \put(47.5,15){\line(-2,1){30}}
        \put(82.5,67){\line(-2,1){30}}
        \put(15.5,33){\line(2,1){66}}
        \put(17.5,31){\line(2,1){66}}
        \put(51.5,16){\line(2,1){30}}
        \put(52.5,14){\line(2,1){30}}
        \put(16.5,68){\line(2,1){30}}
        \put(17.5,66){\line(2,1){30}}
        \put(50,16){\line(0,1){5}}
        \put(50,26){\line(0,1){5}}
        \put(50,36){\line(0,1){5}}
        \put(50,46){\line(0,1){5}}
        \put(50,56){\line(0,1){5}}
        \put(50,66){\line(0,1){5}}
        \put(50,76){\line(0,1){5}}
        \put(15,33){\line(0,1){4}}
        \put(15,42){\line(0,1){4}}
        \put(15,51){\line(0,1){4}}
        \put(15,60){\line(0,1){4}}
        \put(85,33){\line(0,1){4}}
        \put(85,42){\line(0,1){4}}
        \put(85,51){\line(0,1){4}}
        \put(85,60){\line(0,1){4}}

        \put(47,3){$e$}
        \put(35,91){$s_1s_2s_1$}
        \put(90,28){$s_2$}
        \put(90,64){$s_2s_1$}
        \put(0,28){$s_1$}
        \put(-10,64){$s_1s_2$}

\put(170,90){{\rm labels}}        
\put(150,70){\line(1,0){40}}
\put(200,66.5){= $\alpha_1$}
\put(150,52){\line(1,0){40}}
\put(150,48){\line(1,0){40}}
\put(200,46.5){= $\alpha_2$}
\put(150,30){\line(1,0){5}}
\put(158.5,30){\line(1,0){5}}
\put(167,30){\line(1,0){5}}
\put(175.5,30){\line(1,0){5}}
\put(184,30){\line(1,0){5}}
\put(200,26.5){= $\alpha_1+\alpha_2$}

\end{picture}
\end{center}
where $e$ is the identity element and $s_1$, $s_2$ are the simple reflections corresponding to simple roots $\alpha_1$, $\alpha_2$ respectively.
The GKM condition says that the collection of polynomials $(f_w)_{w\in W}$ satisfies the following condition:
if $w$ and $v$ are connected by an edge labeled by $\alpha$ in the GKM graph, then the difference $f_w-f_v$ must be divisible by the polynomial $\alpha$.
For example, the following collection of polynomials satisfies the GKM condition:\\
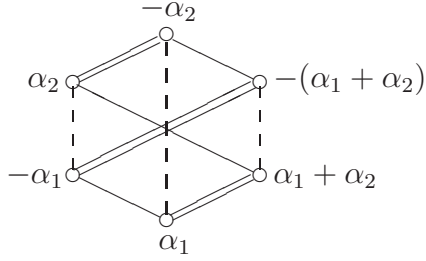
 \begin{center}
\begin{picture}(100,90)
        \put(50,14){\circle{5}}
        \put(50,84){\circle{5}}
        \put(85,31){\circle{5}}
        \put(85,66){\circle{5}}
        \put(15,31){\circle{5}}
        \put(15,66){\circle{5}}

        \put(82.5,32){\line(-2,1){66}}
        \put(47.5,15){\line(-2,1){30}}
        \put(82.5,67){\line(-2,1){30}}
        \put(15.5,33){\line(2,1){66}}
        \put(17.5,31){\line(2,1){66}}
        \put(51.5,16){\line(2,1){30}}
        \put(52.5,14){\line(2,1){30}}
        \put(16.5,68){\line(2,1){30}}
        \put(17.5,66){\line(2,1){30}}
        \put(50,16){\line(0,1){5}}
        \put(50,26){\line(0,1){5}}
        \put(50,36){\line(0,1){5}}
        \put(50,46){\line(0,1){5}}
        \put(50,56){\line(0,1){5}}
        \put(50,66){\line(0,1){5}}
        \put(50,76){\line(0,1){5}}
        \put(15,33){\line(0,1){4}}
        \put(15,42){\line(0,1){4}}
        \put(15,51){\line(0,1){4}}
        \put(15,60){\line(0,1){4}}
        \put(85,33){\line(0,1){4}}
        \put(85,42){\line(0,1){4}}
        \put(85,51){\line(0,1){4}}
        \put(85,60){\line(0,1){4}}

        \put(47,3){$\alpha_1$}
        \put(40,91){$-\alpha_2$}
        \put(90,28){$\alpha_1+\alpha_2$}
        \put(90,64){$-(\alpha_1+\alpha_2)$}
        \put(-10,28){$-\alpha_1$}
        \put(-2,64){$\alpha_2$}


\end{picture}
\end{center}
\end{example}

To describe the GKM condition for $\Hess(S,I)$, we need to investigate the condition $U_\alpha w \cup U_\alpha v\subset \Hess(S,I)$.

\begin{lemma}\label{lemma:GKMcond}
The condition $U_\alpha w\cup U_\alpha v\subset \Hess(S,I)$ is equivalent to the condition 
\[
w^{-1}(\alpha)\in (-I)\cup \Phi ^+ \text{ and } v^{-1}(\alpha)\in (-I)\cup \Phi ^+.
\]
\end{lemma}

\begin{proof}
It is enough to prove that the condition
$U_\alpha w\subset \Hess(S,I)$
is equivalent to the condition 
$w^{-1}(\alpha)\in (-I)\cup \Phi ^+$.
Let $x$ be an arbitrary element of $U_\alpha w$.
Then it is enough to prove that $x\in \Hess(S,I)$ if and only if $w^{-1}(\alpha)\in (-I)\cup \Phi ^+$.
Let $u_\alpha : \CC \rightarrow G$ be a morphism such that $U_\alpha={\rm im}(u_\alpha)$ and write $x=u_\alpha(c)w$ with some $c\in \CC$.
Then
\begin{equation} \label{eq:GKM1}
x\in \Hess(S,I) \Leftrightarrow \mbox{Ad}((u_\alpha (c)w)^{-1})(S)\in H(I)=\mathfrak{b}\oplus (\bigoplus_{\alpha \in I} \mathfrak{g}_{-\alpha}).
\end{equation}
Since $w^{-1}U_\alpha w=U_{w^{-1}(\alpha)}$, we have $w^{-1}u_\alpha (c)^{-1}=u_{w^{-1}(\alpha)} (d) w^{-1}$ for some $d\in \CC$.
Therefore, 
\begin{equation} \label{eq:GKM2}
\mbox{Ad}((u_\alpha (c)w)^{-1})(S)=\mbox{Ad}(u_{w^{-1}(\alpha)} (d))(\mbox{Ad}(w^{-1})(S))\in \mathfrak{t}\oplus \mathfrak{g}_{w^{-1}(\alpha)}
\end{equation}
where the last assertion $\in$ is because $S\in \mathfrak{t}$ and $w\in W$.  It follows from \eqref{eq:GKM1} and \eqref{eq:GKM2} that 
\begin{align*}
x\in \Hess(S,I) \Leftrightarrow w^{-1}(\alpha)\in (-I)\cup \Phi ^+
\end{align*}
and we are done.
\end{proof}

\begin{proof}[Proof of Proposition~\ref{prop:GKM}] 
From Proposition \ref{prop:1-dim orbit} and Lemma~\ref{lemma:GKMcond}, we obtain the following GKM condition for $\Hess(S,I)$: 
\begin{equation*} 
\left\{ (f_w)_{w\in W} \in \bigoplus _{w\in W}H^*(BT) \left |
\begin{array}{ll} f_w-f_v \in (\alpha) \ {\rm if} \ v=s_{\alpha}w, w^{-1}(\alpha )\in (-I)\cup \Phi ^+ \ \\
\text{ and } v^{-1}(\alpha )\in (-I)\cup \Phi ^+ \ {\rm for \ some} \ \alpha \in \Phi ^+ 
\end{array}
\right\}.
\right.
\end{equation*}
When $v=s_\alpha w$, we have 
\[
v^{-1}(\alpha)=(s_\alpha w)^{-1}(\alpha)=w^{-1}s_\alpha(\alpha)=-w^{-1}(\alpha).
\]
Therefore, the conditions $w^{-1}(\alpha )\in (-I)\cup \Phi ^+$ and $v^{-1}(\alpha )\in (-I)\cup \Phi ^+$ above  are equivalent to the condition $w^{-1}(\alpha)\in I\cup (-I)$ when $v=s_\alpha w$.   
We put $\beta=w^{-1}(\alpha)$ when $w^{-1}(\alpha)\in I$, and $\beta=-w^{-1}(\alpha)$ when $w^{-1}(\alpha)\in -I$.  Then $\beta\in I$, $w(\beta)=(\alpha)$ and 
\[
s_\alpha w=w (w^{-1} s_\alpha w)=w s_{w^{-1}(\alpha)}=w s_{\beta}.
\]
Therefore, the GKM condition 
above 
coincides with the GKM condition in Proposition~\ref{prop:GKM}.
\end{proof}


\subsection{$W$-action on $H^*(\Hess(S,I))$}

In this subsection, we define a $W$-action on $H^*(\Hess(S,I))$ using the GKM condition for $\Hess(S,I)$. It is the dot action introduced by Tymoczko in type $A$ (\cite{Ty2}). 

Through the restriction map in \eqref{eq:HessSIRes}, we regard $H^*_T(\Hess(S,I))$ as a submodule of $\bigoplus _{w\in W}H^*(BT)$ and we first define a $W$-action on $\bigoplus _{w\in W}H^*(BT)$ as follows:
\begin{align} \label{eq:W-act}
(u\cdot f)_w:=u(f_{u^{-1}w}) \quad\text{for  $u\in W$ and $f\in \bigoplus _{w\in W}H^*(BT)$}
\end{align}
where $f_w$ is the $w$-component of $f$ and the $W$-action on $H^*(BT)$ is the one induced from the $W$-action on the character $\hat T$ of $T$, see \S\ref{sect:RegNil}.

\begin{lemma}
The $W$-action in \eqref{eq:W-act} preserves $H^*_T(\Hess(S,I))$.
\end{lemma}

\begin{proof}
Let $f\in H^*_T(\Hess(S,I))$ and $u\in W$. 
We denote the restriction image of $f$ to $H^*_T(w)=H^*(BT)$ $(w\in W=\Hess(S,I)^T)$ by $f|_w$ as before.  
Because of Proposition~\ref{prop:GKM}, it is enough to prove that if $\alpha\in I$, then
\[
(u\cdot f)|_w-(u\cdot f)|_{w s_\alpha}\in (w(\alpha)).
\]
Since $f\in H^*_T(\Hess(S,I))$, we have
\[
f|_{u^{-1}w}-f|_{(u^{-1}w) s_\alpha} \in (u^{-1}w(\alpha))
\]
by Proposition~\ref{prop:GKM}.
These together with \eqref{eq:W-act} show 
\begin{align*}
(u\cdot f)|_w-(u\cdot f)|_{w s_\alpha}&=u(f|_{u^{-1}w})-u(f|_{u^{-1}w s_\alpha}) \\
                                      &=u(f|_{u^{-1}w}-f|_{u^{-1}w s_\alpha})\in (w(\alpha))
\end{align*}
and we are done.
\end{proof}

Remember that since $H^{odd}(\Hess(S,I))$ vanishes, 
\[
H^*_T(\Hess(S,I))=H^*(BT)\otimes H^*(\Hess(S,I)) \quad\text{as $H^*(BT)$-modules}. 
\] 
This means that 
\[
H^*(\Hess(S,I))=H^*_T(\Hess(S,I))/H^{>0}(BT)H^*_T(\Hess(S,I))
\]
where $H^{>0}(BT)$ denotes the positive degree part of $H^*(BT)$. 
Since $H^*_T(\Hess(S,I))$ is regarded as an $H^*(BT)$-module through the projection map $ET\times_T \Hess(S,I)\to ET/T=BT$, $H^*(BT)$ in $H^*_T(\Hess(S,I))$ maps to the diagonal part of $\bigoplus _{w\in W}H^*(BT)$ by the restriction map in \eqref{eq:HessSIRes}.  
On the other hand, the $W$-action in \eqref{eq:W-act} preserves the diagonal part.
Therefore, the $W$-action on $H^*_T(\Hess(S,I))$ induces a $W$-action on $H^*(\Hess(S,I))$.

Remember that we have an element $e^T(L_\alpha)\in H^2_T(G/B)$, see \S\ref{sect:RegNil}.
Through the restriction map $H^2_T(G/B)\to H^2_T(\Hess(S,I))$, we may think of $e^T(L_\alpha)$ as an element of $H^2_T(\Hess(S,I))$.
They are same if we regard them as elements of $\bigoplus_{w\in W}H^2(BT)$ through the restriction map in \eqref{eq:HessSIRes}.  

\begin{lemma}\label{lemma:W-inv}
The element $e^T(L_\alpha)\in H^2_T(\Hess(S,I))\subset \bigoplus_{w\in W}H^2(BT)$ is $W$-invariant.  Hence $e(L_\alpha)\in H^2(\Hess(S,I))$ is also $W$-invariant.
\end{lemma}

\begin{proof}
From \eqref{eq:W-act} and Lemma~\ref{lemma:LocFlag}, we have
\[
(u\cdot e^T(L_\alpha))|_w=u(e^T(L_\alpha)|_{u^{-1}w})=u(u^{-1}w(\alpha))=w(\alpha)=e^T(L_\alpha)|_w
\]
for all $u, w \in W$, so $u\cdot e^T(L_\alpha)=e^T(L_\alpha)$ for all $u\in W$.
\end{proof}

\begin{rem} \label{rema:WonFlag}
Since $e(L_{\alpha})\in H^2(G/B)$ for $\alpha\in\hat T$ generate $H^*(G/B)$ as an 
$\R$-algebra, we have $H^*(G/B)^W=H^*(G/B)$ from Lemma~\ref{lemma:W-inv}.
However, the $W$-action on $H^*(\Hess(S,I))$ is nontrivial in general.  
\end{rem}

We conclude this subsection with the following proposition. 

\begin{prop}\label{prop:W-invFinGen}
The ring $H^*_T(\Hess(S,I))^W$ of $W$-invariants is a polynomial ring over $\R$ generated by $e^T(L_{\alpha})$ for $\alpha\in \hat T$.  Moreover  
$H^*(\Hess(S,I))^W$ is generated by $e(L_{\alpha})$ for $\alpha\in \hat T$ as an $\R$-algebra.
\end{prop}

\begin{proof}
Let $e$ be the identity element in $W$.
We will prove that the restriction map to the point $e$
\[
H^*_T(\Hess(S,I))^W\subset H^*_T(\Hess(S,I))\rightarrow H^*_T(e)=H^*(BT)
\]
is an isomorphism.
First, we prove the injectivity. Let $f\in H^*_T(\Hess(S,I))^W$ such that $f|_e=0$. From \eqref{eq:W-act},
\[
f|_w=(u\cdot f)|_w=u(f|_{u^{-1}w})\quad\text{for all $u, w\in W$.}
\]
Taking $u=w$, we have 
\[
f|_w=w(f|_{e})=0\quad\text{for all $w\in W$},
\]
so we obtain $f=0$. 
Next, we prove the surjectivity.
We note that $e^T(L_\alpha)$ lies in $H^2(\Hess(S,I))^W$ by Lemma~\ref{lemma:W-inv} and $e^T(L_\alpha)|_e=\alpha$ by Lemma~\ref{lemma:LocFlag}.
Since $H^*(BT)$ is generated by $\alpha\in \hat T$ through the identification $H^2(BT;\Z)=\hat T$, this proves the surjectivity. 

Since the restriction map $H^*_T(\Hess(S,I))\to H^*(\Hess(S,I))$ is surjective and $W$ is a finite group, its restriction to the $W$-invariants is also surjective.
This together with the former statement implies the latter statement in the proposition. 
\end{proof}

\subsection{Gysin map}
Equivariant Gysin map plays an important role in the following argument and we first review basic facts on it needed later.
We refer the reader to \cite{F} and \cite{Ka} for more details.
The reference \cite{F} discusses in the category of algebraic geometry while \cite{Ka} discusses in the category of topology but both are essentially same.  
We follow the description of \cite{F}.  

Let  $G$ be a linear algebraic group acting on nonsingular algebraic varieties $X$ and $Y$, and let $h\colon X\to Y$ be a proper $G$-equivariant morphism
(since $X$ and $Y$ are regular semisimple Hessenberg varieties or a point in our case, they are compact and hence the properness of $h$ is satisfied for any $h$).
Then there is an {\bf equivariant Gysin map}
\[
h_!^G\colon H^{*}_G(X)\to H^{*+2d}_G(Y)
\]
where $d=\dim_\CC Y-\dim_\CC X$.  
Here are two properties of the equivariant Gysin map used later. 
\begin{enumerate}
\item[(P1)] If $h$ is a closed $G$-equivariant embedding, then its normal bundle $\nu$ becomes a $G$-equivariant vector bundle on $X$,
and the composition $h^*\circ h_!^G\colon H_G^*(X)\to H_G^{*+2d}(X)$ is multiplication by the equivariant Euler class $e^G(\nu)$ of $\nu$.  
\item[(P2)] The map $h_!^G$ reduces to the (ordinary) Gysin map $h_!\colon H^*(X)\to H^{*+2d}(Y)$ through the restriction map from equivariant cohomology to ordinary cohomology
and $h_!$ sends the cofundamental class of $X$ to that of $Y$.\footnote{This is because $h_!$ is defined as the composition $D_Y^{-1}\circ h_*\circ D_X$ where $D_X$ and $D_Y$ are Poincar\'e duality maps of $X$ and $Y$ and $h_*\colon H_*(X)\to H_*(Y)$.}
\end{enumerate}

Now we return to our previous setting. 
Let $\alpha\in \Phi^+\backslash I$ such that $I\cup\{\alpha\}$ is a lower ideal.
Then $\Hess(S,I)$ is a $T$-invariant complex codimension one submanifold of $\Hess(S,I\cup\{\alpha\})$
and the inclusion map $j_\alpha\colon \Hess(S,I)\hookrightarrow \Hess(S,I\cup\{\alpha\})$ induces an equivariant Gysin map  
\[
{j_\alpha^T}_!\colon H^*_T(\Hess(S,I))\to H_T^{*+2}(\Hess(S,I\cup\{\alpha\}))
\]
which raises cohomology degrees by two.
On the other hand, we have a complex $T$-line bundle $L_\alpha$ over $G/B$ (see \S\ref{sect:RegNil})
and regard its equivariant Euler class $e^T(L_\alpha)$ as an element of $\bigoplus_{w\in W}H^2(BT)$ through the restriction to the $T$-fixed point set $W$.
We consider the following diagram:
\begin{equation} \label{eq:GysinEuler}
\begin{CD}
H^*_T(\Hess(S,I)) @>{j_\alpha^T}_!>> H^{*+2}_T(\Hess(S,I\cup\{\alpha\}))\\
@VVV  @VVV\\
\bigoplus_{w\in W}H^*(BT)@>\times e^T(L_{-\alpha}) >> \bigoplus_{w\in W}H^{*+2}(BT)
\end{CD} 
\end{equation}
where the vertical maps are restrictions to the $T$-fixed point set $W$ (so they are injective) and the bottom horizontal map is multiplication by $e^T(L_{-\alpha})=-e^T(L_{\alpha})\in \bigoplus_{w\in W}H^2(BT)$. 

\begin{lemma} \label{lemm:Com}
The diagram \eqref{eq:GysinEuler} is commutative.  Hence the equivariant Gysin map ${j_\alpha^T}_!$ is $W$-equivariant since $e^T(L_{-\alpha})$ is $W$-invariant by Lemma~\ref{lemma:W-inv}.  
\end{lemma}

\begin{proof}
It follows from Property (P1) of equivariant Gysin maps that   
\begin{equation} \label{eq:normal}
{j_\alpha}^*({j_\alpha^T}_!(x))=xe^T(\nu_\alpha) \quad\text{for $x\in H^*_T(\Hess(S,I))$}
\end{equation}
where ${j_\alpha}^*\colon H^*_T(\Hess(S,I\cup\{\alpha\}))\to H^*_T(\Hess(S,I))$ is the restriction homomorphism and $\nu_\alpha$ is the $T$-equivariant normal bundle of $\Hess(S,I)$ in $\Hess(S,I\cup\{\alpha\})$.
It follows from \eqref{eq:TwHess} that $\nu_\alpha$ restricted to $w$ is the $T$-module determined by $w(-\alpha)$ and this agrees with the restriction of $L_{-\alpha}$ to $w$ by Lemma~\ref{lemma:LocFlag}.
Therefore, $e^T(\nu_\alpha)|_w=e^T(L_{-\alpha})|_w$ for all $w\in W$ and hence $e^T(\nu_\alpha)=e^T(L_{-\alpha})$.
This together with \eqref{eq:normal} and the injectivity of $j_\alpha^*$ 
implies the commutativity of the diagram \eqref{eq:GysinEuler}.   
\end{proof}

The inclusion map $j_\alpha$ also induces a Gysin map in ordinary cohomology: 
\begin{equation} \label{eq:GysinOrd}
{j_\alpha}_!\colon H^*(\Hess(S,I))\to H^{*+2}(\Hess(S,I\cup\{\alpha\})
\end{equation}
which also raises cohomology degrees by two. 

\begin{prop}\label{prop:W-Gysin}
The Gysin map ${j_\alpha}_!$ in \eqref{eq:GysinOrd} is $W$-equivariant.  In particular, it maps $H^*(\Hess(S,I))^W$ to $H^{*+2}(\Hess(S,I\cup \{\alpha\}))^W$. 
\end{prop}

\begin{proof}
As mentioned in Property (P2) of equivariant Gysin maps,
${j_\alpha^T}_!$ reduces to the (ordinary) Gysin map ${j_\alpha}_!$ through the restriction map from equivariant cohomology to ordinary cohomology.
Therefore, ${j_\alpha}_!$ is $W$-equivariant because so is ${j_\alpha^T}_!$ by Lemma~\ref{lemm:Com}. 
\end{proof}

\subsection{Poincar\'e duality on $H^*(\Hess(S,I))^W$}
We will prove that $H^*(\Hess(S,I))^W$ is a Poincar\'e duality algebra of socle degree $2|I|$ where $|I|=\dim_\CC\Hess(S,I)$ by Theorem~\ref{theorem:regsemi}.  

We have a pairing defined by cup product composed with the evaluation map on the fundamental class of $\Hess(S,I)$: 
\begin{equation}\label{PairHessS0}
H^{2k}(\Hess(S,I))\times H^{2|I|-2k}(\Hess(S,I))\rightarrow H^{2|I|}(\Hess(S,I))\to \R
\end{equation}
This pairing is non-degenerate since $\Hess(S,I)$ is a compact smooth equidimensional oriented manifold.  
If the $W$-action on $H^*(\Hess(S,I))$ is induced from a $W$-action on $\Hess(S,I)$ preserving the orientation,
then the fundamental class of $\Hess(S,I)$ is invariant under the induced $W$-action so that the evaluation map in \eqref{PairHessS0} is $W$-invariant.
This means that the pairing \eqref{PairHessS0} is $W$-invariant and its restriction to $H^*(\Hess(S,I))^W$ is still non-degenerate
and this will imply that $H^*(\Hess(S,I))^W$ is a Poincar\'e duality algebra.   
However, since our $W$-action on $H^*(\Hess(S,I))$ is defined algebraically using GKM theory, this argument does not work and we need to check the $W$-invariance of the evaluation map in a different way.  

\begin{prop}\label{prop:PDAHessS}
$H^{*}(\Hess(S,I))^W$ is a Poincar\'e duality algebra of socle degree $2|I|$.
\end{prop}

\begin{proof}
The map $\rho : \Hess(S,I)\rightarrow \{pt\}$ induces the $T$-equivariant Gysin map 
\[
\rho^T_{!}: H^*_T(\Hess(S,I))\rightarrow H^{*-2|I|}_T(pt)=H^{*-2|I|}(BT)
\]
which decreases cohomology degrees by $2|I|=\dim_\R \Hess(S,I)$. The Atiyah-Bott-Berline-Vergne formula (\cite{AB}) tells us that 
\begin{align*}
\rho^T_{!}(f)=\sum_{w\in W} \frac{f|_w}{e^T(T_w \Hess(S,I))}
\end{align*}
where $T_w \Hess(S,I)$ denotes the tangent space of $\Hess(S,I)$ at $w$.  It follows from \eqref{eq:TwHess} that 
\begin{align*}
e^T(T_w \Hess(S,I))=\prod_{\alpha\in -I} w(\alpha),
\end{align*}
so we have 
\[
\rho^T_{!}(u\cdot f)=\sum_{w\in W} u(f|_{u^{-1}w})\prod_{\alpha\in -I} \frac{1}{w(\alpha)}
                   =\sum_{v\in W} u(f|_v)\prod_{\alpha\in -I} \frac{1}{u(v(\alpha))}
                   =u(\rho^T_{!}(f))
\]
for all $u\in W$, proving the $W$-equivariance of $\rho^T_{!}$.  

Since $\rho^T_!$ reduces to the (ordinary) Gysin map 
\[
\rho_{!}: H^*(\Hess(S,I))\rightarrow H^{*-2|I|}(pt)=H^{*-2|I|}(BT)/H^{>0}(BT)
\]
through the map from equivariant cohomology to ordinary cohomology and the $W$-action on $H^*(pt)$ is trivial, $\rho_!$ is $W$-invariant.  
Since $H^{*-2|I|}(pt)=0$ unless $*=2|I|$, $\rho_!$ can be nontrivial only when $*=2|I|$.  In fact, $\rho_!$ is the evaluation map on the fundamental class of $\Hess(S,I)$ as is well-known.  Therefore the non-degenerate pairing \eqref{PairHessS0} is $W$-invariant. 
%

Since $\dim_\R H^0(\Hess(S,I))^W=1$ by Proposition~\ref{prop:W-invFinGen}, the existence of the $W$-invariant non-degenerate pairing \eqref{PairHessS0} implies that  $\dim_\R H^{2|I|}(\Hess(S,I))^W=1$  
and $H^*(\Hess(S,I))^W$ is a Poincar\'e duality algebra of socle degree $2|I|$.  
\end{proof}

\section{Proof of Theorem \ref{ssmain}}

Remember that $\CR=\mbox{Sym}(\hat T\otimes \R)$ (see \S\ref{subsect:ringR}).  The map 
\begin{equation}\label{eq:psiI3}
\psi_I\colon \CR\to H^*(\Hess(S,I))^W
\end{equation}
sending $\alpha\in\hat T$ to $e(L_\alpha)\in H^2(\Hess(S,I))^W$ is surjective by Proposition~\ref{prop:W-invFinGen}.  We define 
\[
\mathfrak{s}(I):=\ker\psi_I,
\]
so that we have 
\begin{equation} \label{eq:siso}
\CR/\mathfrak{s}(I)\cong H^*(\Hess(S,I))^W.
\end{equation}
The following theorem implies Theorem~\ref{ssmain} in the Introduction. 

\begin{theorem}
$\mathfrak{s}(I)=\mathfrak{a}(I)$ for any lower ideal $I$. 
\end{theorem}


\begin{proof}
When $I=\Phi^+$, $\Hess(S,I)=G/B$ and $H^*(G/B)^W=H^*(G/B)$ by Remark~\ref{rema:WonFlag}. Hence
$\mathfrak{s}(\Phi^+)=(\CR^W_+)$ by Borel's theorem while $\mathfrak{a}(\Phi^+)=(\CR^W_+)$ by Theorem~\ref{coinv}, so we have $\mathfrak{s}(\Phi^+)=\mathfrak{a}(\Phi^+)$. Therefore, it suffices to prove 
\begin{equation} \label{eq:scolon}
\mathfrak{s}(I)= \mathfrak{s}(I'):\alpha
\end{equation}
for $\alpha \in \Phi^+ \setminus I$ such that $I'=I \cup\{\alpha\}$ is a lower ideal because the ideals $\mathfrak{a}(I)$'s satisfy the same identities.   

The Gysin map ${j_\alpha}_!$ sends $H^*(\Hess(S,I))^W$ to $H^{*+2}(\Hess(S,I'))^W$ by Proposition~\ref{prop:W-Gysin} and is nontrivial
since it maps the cofundamental class of $\Hess(S,I)$ to that of $\Hess(S,I')$ and those cofundamental classes are in the $W$-invariants.
The Gysin map ${j_\alpha}_!$ can be regarded as a map from $\CR/\mathfrak{s}(I)$ to $\CR/\mathfrak{s}(I')$ by \eqref{eq:siso}
and the commutativity of the diagram \eqref{eq:GysinEuler} implies that the map is just multiplication by $-\alpha$.
Here both $\CR/\mathfrak{s}(I)$ and $\CR/\mathfrak{s}(I')$ are Poincar\'e duality algebras by Proposition~\ref{prop:PDAHessS},
so the desired fact \eqref{eq:scolon} follows from Lemma~\ref{colonideal}. 
\end{proof}



\section{Explicit description of $D(\A_I)$ and $\mathfrak a(I)$ for types $A,B,C$ and $G$}

Throughout this section,
$V$ is an $n$-dimensional real vector space with an inner product, $x_1,\dots,x_n$ form an  orthonormal basis of $V^*$, $R:=\mathrm{Sym}(V^*)=\RR[x_1,\dots,x_n]$, and $\partial_i=\partial/\partial x_i$ for all $i$.
It is an interesting but challenging problem to find an explicit description of cohomology rings of all (regular) Hessenberg varieties as quotients of polynomial rings.
While to obtain an explicit description of a cohomology ring is difficult in general,
Theorem \ref{nilpotentmain} makes this problem quite tractable in the case of regular nilpotent Hessenberg varieties.
Indeed, it guarantees that, to find such a description, it suffices to compute a basis of $D(\A_I)$ and to find an explicit description of $\CR/\mathfrak a(I)$. 
Recall that $\A_I$ is an hyperplane arrangement in $\mathfrak t$, $\mathfrak a(I)$ is an ideal of $\CR=\mathrm{Sym}(\mathfrak t^*)$ and the set of positive roots $\Phi^+$ lives in $\CR$ (see Section 3).
In this section,
we will identify $\mathfrak t$ with (a subspace of) $V$, and find an explicit description of the ring $\mathcal R/\mathfrak a(I)$ as a quotient of the polynomial ring $R=\R[x_1,\dots,x_n]$ for types $A,B,C$ and $G$.

We refer the readers to \cite[III.\ \S 12]{Hum1} for a concrete construction of a root system, 
which will be used in this section. 

\subsection{Background}

We recall the result in \cite{AHHM}
which gives an explicit description of the cohomology rings of regular nilpotent Hessenberg varieties in type $A$.
A Hessenberg function (of type $A_{n-1}$) is a function $h: \{1,2,\ldots , n \} \to \{1,2,\ldots , n \}$ satisfying the following two conditions:
\begin{enumerate}
\item $i \leq h(i)$\ \ for $i=1,2,\ldots, n$, and
\item $h(1)\leq h(2)\leq \ldots \leq h(n)$.
\end{enumerate}
To a Hessenberg function $h$, the lower ideal $I$ in the positive roots $\Phi^+_{A_{n-1}}$ of 
type $A_{n-1}$ is defined by 
\[
I=\{x_i-x_j \mid 1\leq i \leq n-1,\ i < j \leq h(i) \}
\]
and this correspondence gives a bijection between Hessenberg functions of type $A_{n-1}$ and lower ideals 
in $\Phi^+_{A_{n-1}}$.

\begin{rem}
In type $A_{n-1}$, the regular nilpotent element $N$ can be regarded as the nilpotent matrix of size $n$ with one Jordan block and if $h$ is the Hessenberg function associated to a lower ideal $I$, then one can see that the regular nilpotent Hessenberg variety $\Hess(N,I)$ consists of the following full flags in $\CC^n$:
$$ \{(V_1 \subsetneq V_2 \subsetneq \cdots \subsetneq V_n=\CC^n) \mid NV_i \subset V_{h(i)} \ {\rm for} \ i=1,2,\ldots,n \}.$$
\end{rem}

For any non-negative integers $i,j$ with $ 1 \leq i \leq j \leq n$,
we define the polynomials $ f^A_{i,j} \in R=\R[x_1,\dots,x_n]$ by
\begin{equation*}
\textstyle
f^A_{i,j} := \sum_{k=1}^i \left ( \prod_{\ell=i+1}^{j} (x_k-x_\ell) \right) x_k 
\end{equation*}
with the convention $\prod_{\ell=i+1}^{i} (x_k-x_\ell)=1$.
(Note that $\check{f}_{i,j}$ in \cite{AHHM} is our $f_{j,i}^A$.)

\begin{theorem}[\cite{AHHM}]
\label{typeAcohomology_murai}
Let $I\subset \Phi^+_{A_{n-1}}$ be a lower ideal and $h$ the corresponding Hessenberg function. Then 
$$
H^*(\Hess(N,I)) \cong \R[x_1,\cdots,x_n]/(f^A_{1,h(1)},\dots,f^A_{n,h(n)}).$$
\end{theorem}
The polynomials $f^A_{i,j}$ $(1\le i\le j\le n)$ were originally defined recursively. 
Indeed, they satisfy the initial conditions
\begin{equation}\label{eq:fi,i}
f_{i,i}^A = x_1+\cdots+x_i \quad {\rm for} \ i=1,\ldots, n, 
\end{equation}
and the recursive formula
\begin{equation}\label{eq:fi,j}
f_{i,j}^A= f_{i-1,j-1}^A +(x_i - x_j) f_{i,j-1}^A \quad \text{ for $1 \leq i < j \leq n$},
\end{equation}
where we take the convention $f_{0,*} = 0$ for any $*$. 
We put $f_{i,j}^A$ in the $(i,j)$ entry of an $n\times n$ matrix. 
Then \eqref{eq:fi,j} means that $f_{i,j}^A$ is determined by ``its northwest and west'':

\smallskip
\begin{center}
\begin{picture}(80,40)
\put(0,0){\framebox(40,20){$f_{i,j-1}^A$}}
\put(0,20){\framebox(40,20){$f_{i-1,j-1}^A$}}
\put(40,0){\framebox(40,20){$f_{i,j}^A$}}
\put(40,20){\framebox(40,20){$\cdots$}}
\end{picture}
\end{center}

\smallskip
\noindent
The polynomial $x_i-x_j$ in \eqref{eq:fi,j} can be regarded as a positive root in $\Phi^+_{A_{n-1}}$. As observed in Example \ref{ex: lower ideal of A3, B3, C3}, it is natural to arrange elements in $\Phi^+_{A_{n-1}}$ in a strict upper triangular $n\times n$ matrix shown below, so $x_i-x_j$ is naturally associated to the $(i,j)$ entry from the Lie-theoretical viewpoint:

\begin{center}
\begin{picture}(350,150)
\put(300,120){\framebox(50,20){\tiny $x_1-x_n$}} 
\put(170,127){$\cdots$}
\put(220,127){$\cdots$} 
\put(50,120){\framebox(50,20){\tiny $x_1-x_3$}}
\put(0,120){\framebox(50,20){\tiny $x_1-x_2$}} 
\put(300,100){\framebox(50,20){\tiny $x_2-x_n$}} 
\put(220,107){$\cdots$} 
\put(100,100){\framebox(50,20){\tiny $x_2-x_4$}}
\put(50,100){\framebox(50,20){\tiny $x_2-x_3$}}

\put(120,85){$\ddots$} 
\put(170,85){$\ddots$} 

\put(300,60){\framebox(50,20){\tiny $x_i-x_n$}}
\put(270,67){$\cdots$} 
\put(200,60){\framebox(50,20){\tiny $x_i-x_{i+2}$}}
\put(150,60){\framebox(50,20){\tiny $x_i-x_{i+1}$}}
\put(220,45){$\ddots$} 
\put(270,45){$\ddots$} 

\put(325,45){$\vdots$} 
\put(325,85){$\vdots$} 

\put(300,20){\framebox(50,20){\tiny $x_{n-2}-x_n$}}
\put(250,20){\framebox(50,20){\tiny $x_{n-2}-x_{n-1}$}} 
\put(300,0){\framebox(50,20){\tiny $x_{n-1}-x_n$}}
\end{picture}
\end{center}
\medskip
It turns out that this observation applied to types $B$, $C$, and $G$ with a little modification produces the desired results.

\subsection{Type $A_{n-1}$}
We first consider type $A_{n-1}$. 
We will identify $\mathfrak t$ with the hyperplane in $V$ defined by the linear form $x_1+ \cdots+x_n$.
In this setting,
$\CR=\mathrm{Sym}(\mathfrak t^*)=R/(x_1+\cdots+x_n)$
and 
$$\Phi_{A_{n-1}}^+=\{x_i-x_j \in \CR\mid 1 \leq i <j \leq n\}.$$
Note that we identify $f \in R$ with its image $f$ modulo $\sum_{i=1}^n x_i$ in $\CR$, and 
the same $f$ denotes both. 

The proof of Theorem~\ref{typeAcohomology_murai} 
given in \cite{AHHM} is based on a careful analysis of the structure of $H^*_S(\Hess(N,h))$
and needs some technical computations.
Below we show that Theorem \ref{typeAcohomology_murai} can be proved quite easily
if we use Theorem~\ref{nilpotentmain}.

Let $\overline \partial=\partial_1+\cdots+\partial_n$.
For $ 1 \leq i \leq j \leq n$,
define
$$\psi^A_{i,j} := \sum_{k=1}^i \left( \prod_{\ell=i+1}^{j} (x_k-x_\ell)\right)\big(\partial_k-\frac 1 n \overline \partial \big) \in \Der \CR=\CR \otimes \mathfrak{t},$$
where
$\mathfrak{t}$ is identified with the $\R$-linear space $\{\sum_{i=1}^n a_i \partial_i\mid \sum_{i=1}^n a_i=0\}$.

\begin{prop}
\label{typeAbasis_murai}
Let $I \subset \Phi^+_{A_{n-1}}$ be a lower ideal and $h$ be the corresponding Hessenberg function.
Then $\psi_{1,h(1)}^A,\dots,\psi_{n-1,h(n-1)}^A$
is an $\CR$-basis of $D(\A_I)$.
\end{prop}

\begin{proof}
Since each $\psi_{i,h(i)}^A$ is an $\CR$-linear combination of $\partial_{1},\dots,\partial_{i}$ 
and the coefficient of $\partial_{i}$ in $\psi_{i,h(i)}^A$ is non-zero modulo $\overline \partial$,
$\psi^A_{1,h(1)},\dots,\psi^A_{n-1,h(n-1)}$ are $\CR$-independent.
Also $\sum_{i=1}^{n-1} \deg \psi_{i,h(i)}^A=\sum_{i=1}^{n-1} (h(i)-i)=|I|=|\A_I|$.
Then, by Saito's criterion (Theorem \ref{Saito}),
what we must prove is that each $\psi_{i,h(i)}^A$ is contained in $D(\mathcal A_I)$.

Let $x_p-x_q \in I$, so $p<q\le h(p)$. Observe that $\overline \partial (x_p-x_q)=0$.
We prove
\begin{align}
\label{typeAeq_murai}
\psi_{i,h(i)}^A (x_p-x_q ) \in (x_p-x_q ) \qquad\text{for all $1 \leq i \leq n-1$.}
\end{align}
One can easily see that 
\[
\psi_{i,h(i)}^A(x_p-x_q)=\begin{cases} 0 \quad &\text{if $i<p$},\\
\prod_{\ell=i+1}^{h(i)} (x_p-x_\ell) \quad &\text{if $p\le i<q$},\\
\prod_{\ell=i+1}^{h(i)} (x_p-x_\ell) -\prod_{\ell=i+1}^{h(i)} (x_q-x_\ell)\quad &\text{if $q\le i$}.\end{cases}
\]
It is obvious that \eqref{typeAeq_murai} holds in the first and third cases above. In the second case, we have $i+1\le q\le h(i)$ since $i< q\le h(p)\le h(i)$ (note that $p\le i$ implies $h(p)\le h(i)$ by condition (2) in the definition of the Hessenberg function $h$); so \eqref{typeAeq_murai} holds in any case. Therefore $\psi_{1,h(1)}^A,\dots,\psi_{n-1,h(n-1)}^A \in D(\A_I)$ as desired.
\end{proof}

Observe that $f^A_{n,h(n)}=f^A_{n,n}=x_1+\cdots+x_n$, and 
 in $\CR$, $f^A_{i,j}=\frac 1 2 \psi^A_{i,j}(x_1^2+ \cdots +x_n^2) 
$ for all $i,j$.
Then,
by Theorem~\ref{nilpotentmain}, 
the following corollary which immediately follows from Proposition~\ref{typeAbasis_murai} proves Theorem \ref{typeAcohomology_murai}.

\begin{cor}
With the same notation as in
Proposition \ref{typeAbasis_murai},
$\mathfrak a(I)$ is the ideal of $\CR=R/(x_1+\cdots+x_n)$ generated by $f^A_{1,h(1)},\dots,f^A_{n-1,h(n-1)}$.
\end{cor}

\begin{rem}
When the lower ideal $I$ consists of the hyperplanes $\ker \alpha$ such that 
$\mbox{ht}(\alpha) \le k$, the corresponding ideal arrangement $\A_{\le k}$ is called a 
\textbf{height subarrangement} of the Weyl arrangement. In 2012, the first author 
learned from Terao the construction of the basis of $D(\A_{\le k})$ for all $k$ when $\Phi$ is 
of type $A$, which 
coincides with that in Proposition \ref{typeAbasis_murai}. At that time no 
relations were found between $D(\A_{\le k})$ and Hessenberg varieties.  
\end{rem}

\subsection{Type $B_n$}
We consider type $B_n$.
Here we identify $\mathfrak t$ with $V$ (in particular, $\CR$ is identified with $R$)
and take
$$\Phi^+_{B_n}=\{ x_i \pm x_j \mid 1 \leq i < j \leq n\} \cup \{ x_i \mid i=1,2,\dots,n\}$$
as the set of positive roots of type $B_n$. 
As observed in Example \ref{ex: lower ideal of A3, B3, C3}, we arrange positive roots in $\Phi^+_{B_n}$ as follows: 
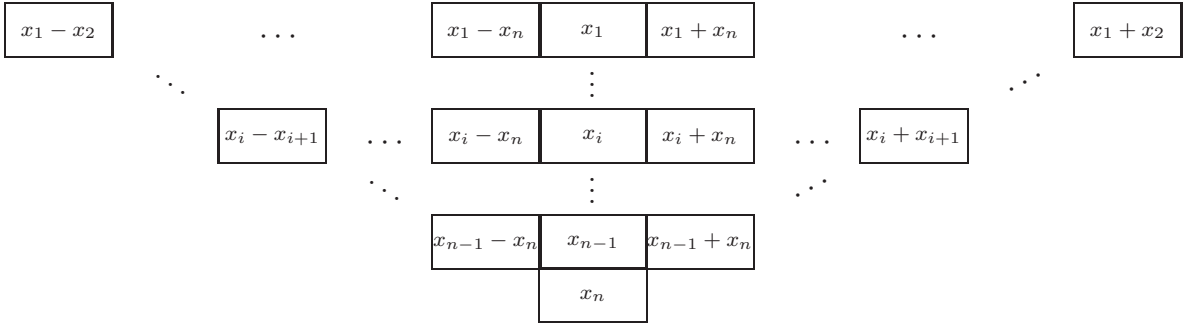
\begin{center}
\begin{picture}(440,120)

\put(400,100){\framebox(40,20){\tiny $x_1+x_2$}}
\put(335,105){$\cdots$}
\put(160,100){\framebox(40,20){\tiny $x_1-x_n$}} 
\put(200,100){\framebox(40,20){\tiny $x_1$}} 
\put(240,100){\framebox(40,20){\tiny $x_1+x_n$}} 
\put(95,105){$\cdots$}
\put(0,100){\framebox(40,20){\tiny $x_1-x_2$}} 
\put(55,85){$\ddots$}
\put(218,85){$\vdots$}
\put(375,85){$\cdot$} 
\put(380,87.5){$\cdot$}
\put(385,90){$\cdot$}

\put(320,60){\framebox(40,20){\tiny $x_i+x_{i+1}$}} 
\put(160,60){\framebox(40,20){\tiny $x_i-x_n$}}
\put(295,65){$\cdots$} 
\put(200,60){\framebox(40,20){\tiny $x_i$}}
\put(135,65){$\cdots$} 
\put(240,60){\framebox(40,20){\tiny $x_i+x_n$}}
\put(80,60){\framebox(40,20){\tiny $x_i-x_{i+1}$}}
\put(135,45){$\ddots$}
\put(218,45){$\vdots$}
\put(295,45){$\cdot$} 
\put(300,47.5){$\cdot$}
\put(305,50){$\cdot$}
\put(240,20){\framebox(40,20){\tiny $x_{n-1}+x_n$}}
\put(200,20){\framebox(40,20){\tiny $x_{n-1}$}}
\put(160,20){\framebox(40,20){\tiny $x_{n-1}-x_n$}}
\put(200,0){\framebox(40,20){\tiny $x_n$}}
\end{picture}
\end{center}

\smallskip
We regard these positive roots as being arranged in the region 
\begin{equation} \label{eq:triangledown}
\triangledown_n:=\{(i,j) \mid 1 \leq i \leq n, \ i+1 \leq j \leq 2n+1-i \}
\end{equation}
of an $n \times 2n$ matrix and denote by $\alpha_{i,j}$ the positive root in the $(i,j)$ entry. Then the $i$-th row  $\alpha_{i,i+1},\alpha_{i,i+2},\dots,\alpha_{i,2n+1-i}$ are 
\[
x_i-x_{i+1},
\dots,
x_i-x_{n},
x_i,
x_i+x_{n},
\dots,
x_i + x_{i+1}.
\]
Note that 
$\alpha\preceq\beta$ for $\alpha,\beta\in\Phi^+_{B_n}$ if and only if $\beta$ is located northeast of $\alpha$, that is, 
\begin{align*}
\alpha_{i,j}\preceq \alpha_{p,q} \ {\rm if \ and \ only \ if} \ i \geq p \ {\rm and} \ j \leq q.
\end{align*}

We define a Hessenberg function (of type $B_n$) to be a function $h:\{1,\dots, n \} \to \{1,\dots, 2n \}$ satisfying the following three conditions:
\begin{enumerate}
\item $i \leq h(i) \leq 2n+1-i$ \ for \ $i=1,2,\ldots, n,$ 
\item if $h(i)\neq 2n+1-i$, then $h(i) \leq h(i+1)$, and 
\item if $h(i)= 2n+1-i$, then $h(k)=2n+1-k$ for $k > i$.
\end{enumerate}
To a Hessenberg function $h$ (of type $B_n$), the lower ideal $I$ in $\Phi^+_{B_n}$ is defined by 
$$I=\{\alpha_{i,j} \mid 1\leq i \leq n, i < j \leq h(i) \}$$
and this correspondence gives a bijection between Hessenberg functions of type $B_n$ and lower ideals in $\Phi^+_{B_{n}}$.

%
%

Motivated by \eqref{eq:fi,i} and \eqref{eq:fi,j}, we define derivations $\psi_{i,j}^B\in\CR\otimes\mathfrak{t}$ by the following initial conditions and recursive formula: 
\begin{align*}
&\psi^B_{i,i} = \partial_1+\cdots+\partial_i \ \ \ \ \ \ \ \ \ \ \ \ \ \ \ {\rm for} \ i=1,\ldots, n, \\
&\psi^B_{i,j} = \psi^B_{i-1,j-1} +\alpha_{i,j} \psi^B_{i,j-1} \ \ \ \ \ \ {\rm for} \ (i,j)\in \triangledown_n,
\end{align*}
where we take the convention $\psi_{0,*} = 0$ for any $*$. Then an elementary computation shows that
\begin{align*}
\textstyle
\psi_{i,j}^B=
\sum_{k=1}^{i} \left( \prod_{\ell=i+1}^j \alpha_{k,\ell} \right) \partial_{{k}}\qquad \text{for $(i,j)\in \triangledown_n \cup \{(i,i) \mid 1 \leq i \leq n \}$}
\end{align*}
with the convention $\prod_{\ell=i+1}^i\alpha_{k,\ell}=1$. 

\begin{rem}
The initial terms $\psi_{1,1}^B,\psi_{2,2}^B,\dots,\psi_{n,n}^B$ are dual to the simple roots $\alpha_{1,2},\alpha_{2,3},\dots,\alpha_{n,n+1}$, i.e., $\psi_{i,i}^B(\alpha_{j,j+1})=\delta_{ij}$ where $\delta_{ij}$ is the Kronecker delta. 
\end{rem}

\begin{example} \label{exam:B3}
In type $B_3$, we have
$$\Phi^+_{B_3}=\{x_1-x_2,
x_1-x_3,
x_1,
x_1+x_3,
x_1+x_2,
x_2-x_3,
x_2,
x_2+x_3,
x_3\}.$$
\begin{center}
\begin{picture}(280,65)

\put(0,00){\framebox(40,20)}
\put(0,20){\framebox(40,20)} 
\put(0,40){\framebox(40,20){$\psi^B_{1,1}$}}
\put(40,00){\framebox(40,20)}
\put(40,20){\framebox(40,20){$\psi^B_{2,2}$}}
\put(40,40){\framebox(40,20){$\psi^B_{1,2}$}}
\put(80,00){\framebox(40,20){$\psi^B_{3,3}$}} 
\put(80,20){\framebox(40,20){$\psi^B_{2,3}$}} 
\put(80,40){\framebox(40,20){$\psi^B_{1,3}$}} 
\put(120,00){\framebox(40,20){$\psi^B_{3,4}$}} 
\put(120,20){\framebox(40,20){$\psi^B_{2,4}$}} 
\put(120,40){\framebox(40,20){$\psi^B_{1,4}$}} 
\put(160,00){\framebox(40,20)} 
\put(160,20){\framebox(40,20){$\psi^B_{2,5}$}} 
\put(160,40){\framebox(40,20){$\psi^B_{1,5}$}} 
\put(200,00){\framebox(40,20)} 
\put(200,20){\framebox(40,20)} 
\put(200,40){\framebox(40,20){$\psi^B_{1,6}$}} 

\end{picture}
\end{center}
Then $\psi^B_{i,j}$ are as follows:  
\begin{align*}
\psi^B_{1,1} &=\partial_{1},\quad
\psi^B_{2,2} =\partial_{1}+\partial_{2},\quad
\psi^B_{3,3} =\partial_{1}+\partial_{2}+\partial_{3},\\
\psi^B_{1,2} &= (x_1-x_2)\partial_{1},\\
\psi^B_{1,3} &= (x_1-x_2)(x_1-x_3)\partial_{1},\\
\psi^B_{1,4} &= (x_1-x_2)(x_1-x_3)x_1\partial_1,\\
\psi^B_{1,5} &= (x_1-x_2)(x_1-x_3)x_1(x_1+x_3)\partial_{1},\\
\psi^B_{1,6} &= (x_1-x_2)(x_1-x_3)x_1(x_1+x_3)(x_1+x_2)\partial_{1},\\
\psi^B_{2,3} &= (x_1-x_3)\partial_{1}+(x_2-x_3)\partial_{2},\\
\psi^B_{2,4} &= (x_1-x_3)x_1\partial_1+(x_2-x_3)x_2\partial_2,\\
\psi^B_{2,5} &= (x_1-x_3)x_1(x_1+x_3)\partial_{1}+(x_2-x_3)x_2(x_2+x_3)\partial_{2},\\
\psi^B_{3,4} &= x_1\partial_1 + x_2\partial_2+x_3\partial_3.
\end{align*}
\end{example}



Recall that, for a lower ideal $I \subset \Phi^+_{B_n}$,
$\A_I$ is the hyperplane arrangement in $V$ defined by linear forms in $I$.
We first prove that $\psi^B_{i,2n+1-i}$ is an element of $D(\mathcal A_{I})$ for any ideal $I \subset \Phi^+_{B_n}$.

\begin{lemma}\label{lemma:DerCondTop}
\label{genflag}
Let $1 \leq i \leq n$. Then
$\psi_{i,2n+1-i}^B (\alpha) \in (\alpha)$ for any $\alpha \in \Phi^+_{B_n}$.
\end{lemma}

\begin{proof}
We note that 
\[
\psi^B_{i,2n+1-i}=\sum_{k=1}^{i} \left( \prod_{\ell=i+1}^{2n+1-i} \alpha_{k,\ell} \right) \partial_{{k}}
=\sum_{k=1}^i\left(\prod_{\ell=i+1}^n(x_k-x_\ell)(x_k+x_\ell)\right)x_k\partial_k.
\]
If $\alpha=x_u \pm x_v$ with $v \geq i+1$ or $\alpha= x_u$, then
$\psi^B_{i,2n+1-i}(\alpha)$ is either zero or $\prod_{\ell=i+1}^{n} (x_{u}-x_\ell)(x_{u}+x_\ell) x_{u}$,
which is contained in the ideal $(\alpha)$.
Suppose $\alpha =x_u \pm x_v$ with $u,v \leq i$.
Then
\begin{align*}
&\psi^B_{i,2n+1-i}(x_u\pm x_v)\\
&= \left( \prod_{\ell=i+1}^n (x_u-x_\ell)(x_u + x_\ell) \right) x_u \pm \left( \prod_{\ell=i+1}^n (x_v-x_\ell)(x_v + x_\ell) \right) x_v.
\end{align*}
The right-hand side of the above equation is contained in $(x_u \pm x_v)$ since if we replace $x_u$ by $\pm x_v$ in $\left( \prod_{\ell=i+1}^n (x_u-x_\ell)(x_u + x_\ell) \right) x_u$, then we obtain the polynomial 
$\pm ( \prod_{\ell=i+1}^n (x_v-x_\ell)(x_v + x_\ell)) x_v$.
\end{proof}

We now prove the main result of this subsection.

\begin{theorem}
\label{BasisTypeBC}
Let $I \subset \Phi^+_{B_n}$ be a lower ideal and let $h$ be the corresponding Hessenberg function.
Then $\psi^B_{1,h(1)},\psi^B_{2,h(2)},\dots,\psi^B_{n,h(n)}$ is an $\CR$-basis of $D(\mathcal A_I)$.
\end{theorem}

\begin{proof}
Since each $\psi_{i,h(i)}^B$ is an $\CR$-linear combination of $\partial_{1},\dots,\partial_{i}$ 
and the coefficient of $\partial_{i}$ in $\psi_{i,h(i)}^B$ is non-zero,
$\psi^B_{1,h(1)},\dots,\psi^B_{n,h(n)}$ are $\CR$-independent.
Also $\sum_{i=1}^{n} \deg \psi_{i,h(i)}^B=\sum_{i=1}^{n} (h(i)-i)=|I|=|\A_I|$.
Then, by Saito's criterion (Theorem \ref{Saito}),
what we must prove is that each $\psi_{i,h(i)}^B$ is contained in $D(\mathcal A_I)$.

Let $\alpha_{p,q} \in I$, so 
\begin{equation} \label{eq:st}
p<q\le h(p).
\end{equation} 
We prove
\begin{align}
\label{9-2_murai}
\psi^B_{i,h(i)}(\alpha_{p,q}) \in (\alpha_{p,q}) \ \ \mbox{ for all } 1 \leq i \leq n.
\end{align}
We fix $i$ and may assume $h(i) < 2n+1-i$ by Lemma \ref{lemma:DerCondTop}.
Then one can see that 
\begin{equation} \label{eq:psiBst}
\psi_{i,h(i)}^B(\alpha_{p,q})=\begin{cases} 0 \quad &\text{if $i<p$},\\
\prod_{\ell=i+1}^{h(i)} \alpha_{p,\ell} \quad &\text{if $p\le i<q$},\\
\prod_{\ell=i+1}^{h(i)} \alpha_{p,\ell} -\prod_{\ell=i+1}^{h(i)} \alpha_{q,\ell}\quad &\text{if $q\le i$}.\end{cases}
\end{equation}
Indeed, \eqref{eq:psiBst} is obvious in the first and third cases.
In the second case, since $p\le i$ and $h(i) < 2n+1-i$, it follows from the definition of Hessenberg function of type $B_n$ that we have $h(p)\neq 2n+1-p$ by condition (3) and hence $h(p)\le h(i)$ by condition (2). Therefore, we obtain
$i<q \leq h(p) \leq h(i)< 2n+1-i,$ where the second inequality follows from \eqref{eq:st}. 
The obtained inequality $i<q < 2n+1-i$ means that $\alpha_{p,q}=x_p\pm x_r$ with $r\ge i+1$ or $x_p$. Since $\psi_{i,h(i)}(x_r)=0$ and $p\le i$, the second identity in \eqref{eq:psiBst} follows. 

We have $i<q \leq h(i)$ in the second case and we note that $\alpha_{p,\ell}-\alpha_{q,\ell}=\alpha_{p,q}$ in the third case, so \eqref{9-2_murai} holds in any case. Therefore $\psi_{1,h(1)}^B,\dots,\psi_{n,h(n)}^B \in D(\A_I)$ as desired.
\end{proof}



\begin{cor}
Let $f_{i,j}^B=\frac 1 2 \psi^B_{i,j}(x_1^2+ \cdots +x_n^2)$. Then, with the same notation as in Theorem \ref{BasisTypeBC},
we have 
$\mathfrak a(I)=(f^B_{1,h(1)},f^B_{2,h(2)},\dots,f^B_{n,h(n)})$,
and hence 
$$
H^*(\Hess(N,I)) \cong \R[x_1,\dots,x_n]/(f^B_{1,h(1)},f^B_{2,h(2)},\dots,f^B_{n,h(n)}).
$$
\end{cor}

\begin{example}
Consider type $B_3$. Let $$I=\{x_1-x_2,x_1-x_3,x_2-x_3,x_2,x_2+x_3,x_3\}.$$
Then
$(h(1),h(2),h(3))=(3,5,4)$, and it follows from Example~\ref{exam:B3} that $\CR/\mathfrak a(I) \cong H^*(\Hess (N,I))$ is isomorphic to the quotient of $\R[x_1,x_2,x_3]$ by the ideal generated by the following three polynomials
\begin{align*}
f^B_{1,3}&=(x_1-x_2)(x_1-x_3)x_1,\\
f^B_{2,5}&=(x_1-x_3)(x_1+x_3)x_1^2+(x_2-x_3)(x_2+x_3)x_2^2,\\
f^B_{3,4}&=x_1^2+x_2^2+x_3^2.
\end{align*}
\end{example}

\subsection{Type $C_n$}

Next, we consider type $C_n$.
We again identify $\mathfrak t$ with $V$,
and take
$$\Phi^+_{C_n}=\{ x_i \pm x_j\mid 1 \leq i < j \leq n\} \cup \{ 2 x_i\mid i=1,2,\dots,n\}
$$
as the set of positive roots of type $C_n$.
As observed in Example \ref{ex: lower ideal of A3, B3, C3}, we arrange positive roots in $\Phi^+_{C_n}$ as follows: \\
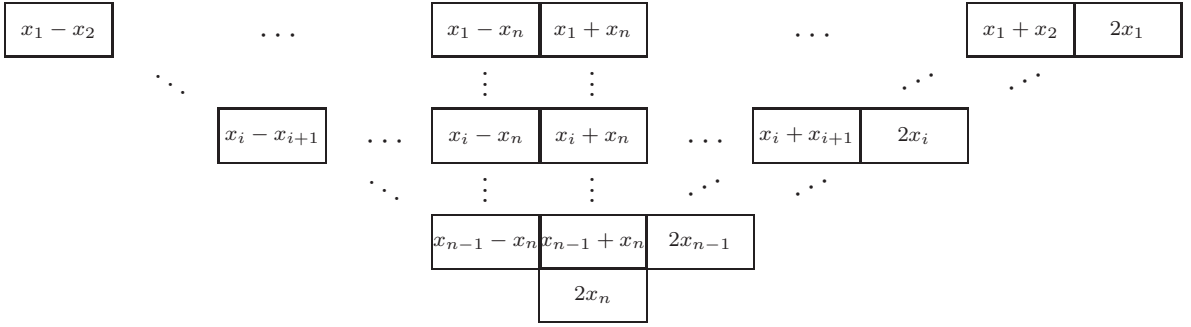
\begin{center}
\begin{picture}(440,100)

\put(400,100){\framebox(40,20){\tiny $2x_1$}}
\put(295,105){$\cdots$}
\put(160,100){\framebox(40,20){\tiny $x_1-x_n$}} 
\put(200,100){\framebox(40,20){\tiny $x_1+x_n$}} 
\put(360,100){\framebox(40,20){\tiny $x_1+x_2$}} 
\put(95,105){$\cdots$}
\put(0,100){\framebox(40,20){\tiny $x_1-x_2$}} 
\put(55,85){$\ddots$}
\put(178,85){$\vdots$}
\put(218,85){$\vdots$}
\put(375,85){$\cdot$} 
\put(380,87.5){$\cdot$}
\put(385,90){$\cdot$}
\put(335,85){$\cdot$} 
\put(340,87.5){$\cdot$}
\put(345,90){$\cdot$}

\put(320,60){\framebox(40,20){\tiny $2x_i$}} 
\put(160,60){\framebox(40,20){\tiny $x_i-x_n$}}
\put(255,65){$\cdots$}
\put(200,60){\framebox(40,20){\tiny $x_i+x_n$}}
\put(135,65){$\cdots$}
\put(280,60){\framebox(40,20){\tiny $x_i+x_{i+1}$}}
\put(80,60){\framebox(40,20){\tiny $x_i-x_{i+1}$}}
\put(135,45){$\ddots$}
\put(178,45){$\vdots$}
\put(218,45){$\vdots$}
\put(295,45){$\cdot$} 
\put(300,47.5){$\cdot$}
\put(305,50){$\cdot$}
\put(255,45){$\cdot$} 
\put(260,47.5){$\cdot$}
\put(265,50){$\cdot$}
\put(240,20){\framebox(40,20){\tiny $2x_{n-1}$}}
\put(200,20){\framebox(40,20){\tiny $x_{n-1}+x_n$}}
\put(160,20){\framebox(40,20){\tiny $x_{n-1}-x_n$}}
\put(200,0){\framebox(40,20){\tiny $2x_n$}}
\end{picture}
\end{center}
\medskip
Similarly to type $B_n$ case, we regard these positive roots as being arranged in the region $\triangledown_n$ (see \eqref{eq:triangledown}). We define $\alpha_{i,j}$ for $(i,j) \in \triangledown_n$ by
\begin{align*}
\alpha_{i,j}=
\begin{cases}
\text{the $(i,j)$ entry} & \mbox{ if } j \ne 2n+1-i,\\
x_i & \mbox{ if } j=2n+1-i,
\end{cases}
\end{align*}
%
%
%
and derivations $\psi_{i,j}^C\in\CR\otimes\mathfrak{t}$ by the following initial conditions and recursive formula: 
\begin{align*}
&\psi^C_{i,i} = \partial_1+\cdots+\partial_i \ \ \ \ \ \ \ \ \ \ \ \ \ \ \ {\rm for} \ i=1,\ldots, n, \\
&\psi^C_{i,j} = \psi^C_{i-1,j-1} +\alpha_{i,j} \psi^C_{i,j-1} \ \ \ \ \ \ {\rm for} \ (i,j)\in \triangledown_n,
\end{align*}
where we take the convention $\psi_{0,*} = 0$ for any $*$ as before. Then an elementary computation shows that 
\begin{align*}
\psi^C_{i,j}=
\begin{cases}
\sum_{k=1}^{i} \left( \prod_{\ell=i+1}^j \alpha_{k,\ell} \right) \partial_{{k}} & \mbox{ if } j \ne 2n+1-i,\\
\sum_{k=1}^{i} \left ( \prod_{\ell=i+1}^n (x_{k}-x_\ell)(x_{k}+x_\ell)\right) x_{k} \partial_{{k}}
& \mbox{ if } j=2n+1-i.
\end{cases}
\end{align*}
for $(i,j)\in \triangledown_n \cup \{(i,i) \mid 1 \leq i \leq n \}$, with the convention $\prod_{\ell=i+1}^i *=1$.

\begin{rem}
(1) $\psi^C_{i,2n+1-i} = \psi^B_{i,2n+1-i}$ for $i=1,\dots,n$. 

(2) The initial terms $\psi^C_{1,1}, \psi^C_{2,2},\dots,\psi^C_{n,n}$ are dual to $\alpha_{1,2},\alpha_{2,3},\dots,\alpha_{n,n+1}$ similarly to type $B_n$ case, but $\alpha_{n,n+1}=x_n$ is not a simple root although the others are simple roots. 
\end{rem}

\begin{example}
In type $C_3$, we have
$$\Phi_{C_3}^+=\{x_1-x_2,
x_1-x_3,
x_1+x_3,
x_1+x_2,
2 x_1,
x_2-x_3,
x_2+x_3,
2 x_2,
2 x_3\}.$$

\begin{center}
\begin{picture}(280,65)

\put(0,00){\framebox(40,20)}
\put(0,20){\framebox(40,20)} 
\put(0,40){\framebox(40,20){$\psi^C_{1,1}$}}
\put(40,00){\framebox(40,20)}
\put(40,20){\framebox(40,20){$\psi^C_{2,2}$}}
\put(40,40){\framebox(40,20){$\psi^C_{1,2}$}}
\put(80,00){\framebox(40,20){$\psi^C_{3,3}$}} 
\put(80,20){\framebox(40,20){$\psi^C_{2,3}$}} 
\put(80,40){\framebox(40,20){$\psi^C_{1,3}$}} 
\put(120,00){\framebox(40,20){$\psi^C_{3,4}$}} 
\put(120,20){\framebox(40,20){$\psi^C_{2,4}$}} 
\put(120,40){\framebox(40,20){$\psi^C_{1,4}$}} 
\put(160,00){\framebox(40,20)} 
\put(160,20){\framebox(40,20){$\psi^C_{2,5}$}} 
\put(160,40){\framebox(40,20){$\psi^C_{1,5}$}} 
\put(200,00){\framebox(40,20)} 
\put(200,20){\framebox(40,20)} 
\put(200,40){\framebox(40,20){$\psi^C_{1,6}$}} 

\end{picture}
\end{center}
Then 
\begin{align*}
\psi^C_{1,4} &= (x_1-x_2)(x_1-x_3)(x_1+x_3)\partial_1,\\
\psi^C_{1,5} &= (x_1-x_2)(x_1-x_3)(x_1+x_3)(x_1+x_2)\partial_{1},\\
\psi^C_{2,4} &= (x_1-x_3)(x_1+x_3)\partial_1+(x_2-x_3)(x_2+x_3)\partial_2,
\end{align*}
and other $\psi_{i,j}^C$ are same as $\psi_{i,j}^B$ in Example~\ref{exam:B3}. 
\end{example}


We define Hessenberg functions of type $C_n$ to be the same as type $B_n$. Since $\psi_{i,2n+1-i}^C=\psi_{i,2n+1-i}^B$ and $\Phi^+_{C_n}$ agrees with $\Phi^+_{B_n}$ up to constant multiples, Lemma \ref{genflag} holds for $\psi_{i,2n+1-i}^C$ and the following theorem can be proved exactly in the same way as in the proof of Theorem \ref{BasisTypeBC}.

\begin{theorem}
\label{BasisTypeC}
Let $I \subset \Phi^+_{C_n}$ be a lower ideal and $h$ be the corresponding Hessenberg function.
Then $\psi^C_{1,h(1)},\psi^C_{2,h(2)},\dots,\psi^C_{n,h(n)}$ is an $\CR$-basis of $D(\mathcal A_I)$.
\end{theorem}

\begin{cor}
Let $f_{i,j}^C=\frac 1 2 \psi^C_{i,j}(x_1^2+ \cdots +x_n^2)$. Then, with the same notation as in Theorem \ref{BasisTypeC},
we have $\mathfrak a(I)=(f^C_{1,h(1)},f^C_{2,h(2)},\dots,f^C_{n,h(n)})$, 
and hence 
$$
H^*(\Hess(N,I)) \cong \R[x_1,\dots,x_n]/(f^C_{1,h(1)},f^C_{2,h(2)},\dots,f^C_{n,h(n)}).
$$
\end{cor}

\subsection{Type $G_2$}
We finally consider type $G_2$.
Let $V$ be the real vector space of dimension $3$ with an inner product,
$x,y,z$ an orthonormal basis of $V^*$, and 
we identify $\mathfrak t$ with 
the hyperplane in $V$ defined by the linear form $x+y+z$.
Then the positive roots $\Phi^+_{G_2}$ of type $G_2$ can be taken as the images of the following polynomials in $\CR=\mbox{Sym}(\mathfrak{t}^*)=\R[x,y,z]/(x+y+z)$: 
$$ x-y,\ \ -x+z,\ \ -y+z,\ \ x-2y+z,\ \ -x-y+2z,\ \ -2x+y+z. $$
We arrange these polynomials in the region 
\begin{equation*} 
\triangledown:=\{(1,2), (1,3), (1,4), (1,5), (1,6), (2,3) \}
\end{equation*}
as follows:
\begin{center}
\begin{picture}(250,40)
\put(0,20){\framebox(50,20){\tiny $x-y$}} 

\put(50,00){\framebox(50,20){\tiny $-2x+y+z$}}
\put(50,20){\framebox(50,20){\tiny $-x+z$}}

\put(100,20){\framebox(50,20){\tiny $-y+z$}} 

\put(150,20){\framebox(50,20){\tiny $x-2y+z$}} 

\put(200,20){\framebox(50,20){\tiny $-x-y+2z$}} 

\end{picture}
\end{center}
\medskip
The images of these entries in $\CR$ are positive roots in $\Phi^+_{G_2}$ and we denote the image of the $(i,j)$ entry by $\alpha_{i,j}$. 
Then 
\[\alpha_{i,j}\preceq \alpha_{p,q} \ {\rm if \ and \ only \ if} \ i \geq p \ {\rm and} \ j \leq q\]
as before.  

We define a Hessenberg function (of type $G_2)$ to be a function $h: \{1,2 \} \to \{1,2,3,4,5,6 \}$ satisfying the following conditions:
\begin{enumerate}
\item $1 \leq h(1) \leq 6$ \ and \ $2 \leq h(2) \leq 3,$ and
\item if $h(1)\geq 3$, then $h(2)=3$.
\end{enumerate}
To a Hessenberg function $h$ (of type $G_2$), the lower ideal $I$ in $\Phi^+_{G_2}$ is defined by 
$$
I=\{\alpha_{i,j} \mid 1 \leq i \leq 2, i<j \leq h(i) \}
$$
and this correspondence gives a bijection between Hessenberg functions of type $G_2$ and lower ideals in $\Phi^+_{G_{2}}$.

Let $\overline \partial=\partial_x+\partial_y+\partial_z$. 
We identify $\mathfrak{t}$ with the $\R$-linear space 
\[
\{a_x\partial_x+a_y\partial_y+a_z\partial_z\mid a_x,a_y,a_z\in\R,\ a_x+a_y+a_z=0\}
\]
and define derivations $\psi^G_{i,j}\in \CR\otimes \mathfrak{t}$ by the following initial conditions and recursive formula:
\begin{align*}
&\psi^G_{1,1} = -\partial_y+\partial_z, \qquad \psi^G_{2,2} = \partial_z- \textstyle \frac 1 3 \overline \partial,  \\
&\psi^G_{i,j} = \psi^G_{i-1,j-1} +\alpha_{i,j} \psi^G_{i,j-1} \ \ \ \ \ \ {\rm for} \ (i,j)\in \triangledown
\end{align*}
with the convention $\psi_{0,*}= 0$ for any $*$. 

\begin{rem}
The initial terms $\psi_{1,1}^G,\psi_{2,2}^G$ are dual to the simple roots $\alpha_{1,2},\alpha_{2,3}$.
\end{rem}

An elementary computation shows that
\begin{align*}
\textstyle
\psi_{i,j}^G=
\begin{cases}
\left(\prod_{\ell=2}^j \alpha_{1,\ell}\right) (-\partial_y+\partial_z) &{\rm if} \ 1=i<j, \\
x\partial_x+y\partial_y+z\partial_z -\frac 1 3(x+y+z)\overline \partial 
&{\rm if} \ (i,j)=(2,3),
\end{cases}
\end{align*}
where the images of $x,y,z$ in $\CR$ are denoted by the same notation respectively.

\begin{theorem}
\label{BasisG2}
Let $I\subset \Phi^+_{G_2}$ be a lower ideal and $h$ be the corresponding Hessenberg function. 
Then $\psi_{1,h(1)}^G, \psi_{2,h(2)}^G$ is an $\CR$-basis of $D(\A_I)$. 
\end{theorem}


\begin{proof}
It is a routine work to check that $\psi^G_{i,h(i)}(\alpha) \in (\alpha)$ for any $\alpha \in I$.
Since $\psi_{1,h(1)}^G$ and $\psi_{2,h(2)}^G$ are $\CR$-independent and $\deg \psi^G_{1,h(1)}+\deg \psi^G_{2,h(2)}=|\A_I|$, Theorem \ref{Saito} proves the desired statement.
\end{proof}

\begin{cor}
Let $f^G_{i,j}=\frac1 2 \psi^G_{i,j}(x^2+y^2+z^2)$.
Then, with the same notation as in Theorem \ref{BasisG2},
we have $\mathfrak{a}(I)=(f_{1,h(1)}^G, f_{2,h(2)}^G)$, and hence 
$$
H^*(\Hess(N,I)) \cong \R[x,y,z]/(x+y+z, f_{1,h(1)}^G, f_{2,h(2)}^G).
$$
\end{cor}

\section{Volume polynomial}

In this section, we will give another type of description of the ideal $\mathfrak{a}(I)$.
In fact, $\mathfrak{a}(I)$ can be described by one homogeneous polynomial 
and we will find it.   
  
For a homogeneous polynomial $P$ of degree $r$ in $n$ variables $x_1,\dots,x_n$, we define the annihilator of $P$ by 
\[
\mbox{Ann}(P):=\{ f\in \R[x_1,\dots,x_n] \mid f(\partial/\partial x_1,\dots,\partial/\partial x_n)(P)=0\}.
\]
One can check that two homogeneous polynomials have the same annihilator if and only if they agree up to a nonzero scalar multiple.
It is well-known (and indeed easy to check) that $\mbox{Ann}(P)$ is a graded ideal of the polynomial ring $\R[x_1,\dots,x_n]$ and the quotient 
\begin{equation*} \label{eq:A(P)}
A(P):=\R[x_1,\dots,x_n]/\mbox{Ann}(P)
\end{equation*}
is a Poincar\'e duality algebra of socle degree $r$.

\begin{rem}
Conversely, any Poincar\'e duality algebra $A$ generated by degree one elements is obtained this way up to isomorphism.
Indeed, if $A$ is of socle degree $r$ and generated by degree one elements $\xi_1,\dots, \xi_n$, then $P_A$ defined by 
\begin{equation} \label{eq:volpoly}
P_A(x_1,\dots,x_n):= \frac{1}{r!}\int_A(x_1\xi_1+\dots+x_n\xi_n)^r,
\end{equation}
is the desired homogeneous polynomial, where $\int_A$ denotes an isomorphism from the $r$-th graded piece of $A$ to $\R$.
Under some situation, the right hand side of \eqref{eq:volpoly} gives the volume of a polytope associated to the element $x_1\xi_1+\dots+x_n\xi_n$ (see \cite[Section 5.3]{F3}, \cite{Kav}). 
\end{rem}

In order to make things clear, we take a coordinate free approach in our case.  
Remember that $G$ is a semisimple linear algebraic group, $\mathfrak{t}$ is the Lie algebra of a compact maximal torus $T_\RR$ of $G$, the Weyl group $W$ of $G$ acts on $\mathfrak{t}$ and its dual $\mathfrak{t}^*$, and 
$\CR=\mbox{Sym}(\mathfrak{t}^*)$.
We set 
\[\MCD:=\mbox{Sym}(\mathfrak{t}).\]  
We regard $\mathfrak{t}$ as derivations on $\CR$ and the algebra $\MCD$ as differential operators on $\CR$ in a natural way.  

We choose a $W$-invariant inner product on $\mathfrak{t}^*$ or $\mathfrak{t}$ which determines an isomorphism between $\mathfrak{t}^*$ and  $\mathfrak{t}$ as $W$-modules.
This isomorphism extends to an equivariant isomorphism between $\CR$ and $\MCD$ as graded algebras with $W$-actions.
Through this isomorphism, we denote the element in $\MCD$ corresponding to $f\in \CR$ by $\partial_f$.
If we choose an orthonormal coordinate system on $\mathfrak{t}$, say $(x_1,\dots,x_n)$, then $f$ can be expressed as a polynomial $f(x_1,\dots,x_n)$ in $x_1,\dots,x_n$ and $\partial_f$ agrees with $f(\partial/\partial x_1,\dots,\partial/\partial x_n)$.
Using this description, one can easily see that
\begin{equation} \label{eq:positive}
\partial_f(f)>0 \quad \text{for any nonzero $f\in \CR$.}
\end{equation} 
For $g\in \CR$ we define
\[\mbox{Ann}(g):=\{ f\in \CR\mid\partial_f(g)=0\}.\]

The roots of $G$ are degree one elements of $\CR$ and we put 
\[P:=\prod_{\alpha\in\Phi^+}\alpha.\]
The following proposition is known as a theorem of Kostant (see \cite{Kav}) and a proof can be found in \cite[Proposition 8.19]{lef}. 

\begin{prop} \label{prop:volG/B}
${\rm Ann}(P)=(\CR_+^W)$.
\end{prop}

For a lower ideal $I\subset \Phi^+$ we put  
\begin{equation} \label{eq:defPI}
P_I:=\partial_{\beta_I}(P)\qquad \text{where}\quad \beta_I=\prod_{\alpha\in\Phi^+\backslash I}\alpha.
\end{equation}
The following is our main result in this section. 

\begin{theorem} \label{theo:volpoly}
${\rm Ann}(P_I)=\mathfrak{a}(I)$ for any lower ideal $I$. 
\end{theorem}

\begin{proof}
$\mbox{Ann}(P)=(\CR_+^W)$ by Proposition~\ref{prop:volG/B} while $\mathfrak{a}(\Phi^+)=(\CR_+^W)$ by Theorem~\ref{coinv}, so $\mbox{Ann}(P_I)=\mathfrak{a}(I)$ when $I=\Phi^+$.  Therefore it suffices to prove that if $I\cup\{\alpha\}$ is a lower ideal for $\alpha\in \Phi^+\backslash I$, then 
\begin{equation} \label{eq:Anncolon}
\mbox{Ann}(P_I)=\mbox{Ann}(P_{I\cup\{\alpha\}})\colon \alpha
\end{equation}
because $\mathfrak{a}(I)$'s satisfy the same identities by Proposition~\ref{inj}.

For any element $f\in \mbox{Ann}(P_I)$, we have 
\[\partial_{\alpha f}(P_{I\cup\{\alpha\}})=\partial_f(\partial_\alpha(P_{I\cup\{\alpha\}}))=\partial_f(P_I)=0,\]
where the second identity follows from \eqref{eq:defPI}.
This proves the inclusion relation $\subset$ in \eqref{eq:Anncolon}.  

To prove the equality in \eqref{eq:Anncolon}, Lemma~\ref{colonideal} says that it suffices to check $\alpha\notin \mbox{Ann}(P_{I\cup\{\alpha\}})$.
Moreover, since $\partial_\alpha(P_{I\cup\{\alpha\}})=P_I$ from \eqref{eq:defPI}, it suffices to prove $P_I\not=0$.  If we set $\alpha_I=\prod_{\alpha\in I}\alpha$, then 
\begin{equation} \label{eq:PInonzero}
\partial_{\alpha_I}(P_I)=\partial_{\alpha_I}(\partial_{\beta_I}(P))=\partial_P(P)>0
\end{equation}
where the first identity follows from \eqref{eq:defPI}, the second is because $\alpha_I\beta_I=P$ and the last positivity follows from \eqref{eq:positive}.
The desired fact $P_I\not=0$ follows from \eqref{eq:PInonzero}. 
\end{proof}

\begin{example}
Let $G$ be a simple linear algebraic group of type $A_2$ and let $\alpha_1,\alpha_2$ be its simple roots.
Then $P=\alpha_1\alpha_2(\alpha_1+\alpha_2)$.
We take  $I=\{\alpha_1,\alpha_2\}$, so $\Phi^+\backslash I=\{\alpha_1+\alpha_2\}$.  We choose a usual $W$-invariant inner product $(\ ,\ )$ on $\mathfrak{t}^*$ such that  $(\alpha_i,\alpha_i)=2$ for $i=1,2$ and $(\alpha_i,\alpha_j)=-1$ for $i\not=j$.
Then it follows from the definition of $\partial_{\alpha_i}$ that 
\[
\partial_{\alpha_i}(\alpha_i)=(\alpha_i,\alpha_i)=2\quad\text{for $i=1,2$},\qquad \partial_{\alpha_i}(\alpha_j)=(\alpha_i,\alpha_j)=-1\quad\text{for $i\not=j$}
\]
and an elementary computation yields  
\[
P_I=\partial_{\alpha_1+\alpha_2}(P)=\partial_{\alpha_1}(P)+\partial_{\alpha_2}(P)=\alpha_1^2+\alpha_2^2+4\alpha_1\alpha_2.
\]
A further elementary computation shows
\[
\mbox{Ann}(P_I)=(\alpha_1^2-\alpha_2^2,\ 2\alpha_1^2+\alpha_1\alpha_2).
\]
\end{example}

\begin{rem}
The polynomial $P_I$ is related to the volume of a Newton-Okounkov body of the regular nilpotent Hessenberg variety $\Hess(N,I)$ in type $A$, see \cite{ADGH}.  
\end{rem}

\section{Hard Lefschetz property and Hodge-Riemann relations}

In this section we observe that the hard Lefschetz property and the Hodge-Riemann relations hold for any regular nilpotent Hessenberg variety $\Hess(N,I)$ of positive dimension, despite the fact that it is a singular variety in general.  
 
Since the regular semisimple Hessenberg variety $\Hess(S,I)$ is non-singular and projective, it is a compact K\"ahler manifold;
so the hard Lefschetz property and the Hodge-Riemann relations hold for the K\"ahler form $\omega$ on it if $|I|=\dim_\CC \Hess(S,I)\ge 1$.
Namely the cohomology class in $H^2(\Hess(S,I))$ determined by $\omega$, denoted $[\omega]$, satisfies the following for $2q\le |I|$:

\medskip
\noindent
{\bf Hard Lefschetz property}. The multiplication by $[\omega]^{|I|-2q}$ defines an isomorphism 
\[H^{2q}(\Hess(S,I))\to H^{2(|I|-q)}(\Hess(S,I)). \]

\medskip
\noindent
{\bf Hodge-Riemann relations}. The multiplication by $[\omega]^{|I|-2q}$ defines a symmetric bilinear form
\[H^{2q}(\Hess(S,I))\times H^{2q}(\Hess(S,I))\to \R,\quad (\xi_1,\xi_2)\mapsto (-1)^q\int[\omega]^{|I|-2q}\cup \xi_1\cup\xi_2\]
that is positive definite when restricted to the kernel of the multiplication map $[\omega]^{|I|-2q+1}\colon H^{2q}(\Hess(S,I))\to H^{2(|I|-q+1)}(\Hess(S,I))$, where $\int$ denotes the evaluation on the fundamental class of $\Hess(S,I)$.  

\medskip
The flag variety $G/B$ is also a non-singular projective variety, so it is a compact K\"ahler manifold as well.  Since $\Hess(S,I)$ is a complex submanifold of $G/B$, the K\"ahler form on $G/B$ restricted to $\Hess(S,I)$ becomes a K\"ahler form on $\Hess(S,I)$.
Therefore one can take $[\omega]$ above as the restriction image of an element in $H^2(G/B)$.
As remarked in Remark~\ref{rema:WonFlag}, the $W$-action on $H^*(G/B)$ is trivial, so that such $[\omega]$ is invariant under the $W$-action.
Therefore, the hard Lefschetz property and the Hodge-Riemann relations above still hold when restricted to the $W$-invariants $H^*(\Hess(S,I))^W$ for such $[\omega]$.
Since $H^*(\Hess(S,I))^W$ is isomorphic to $H^*(\Hess(N,I))$ by Theorem \ref{ssmain}, we obtain the following. 

\begin{theorem}
Any regular nilpotent Hessenberg variety $\Hess(N,I)$ of positive dimension has a non-zero element in $H^2(\Hess(N,I))$ which satisfies the hard Lefschetz property and the Hodge-Riemann relations.  
\end{theorem}

\end{document}